\documentclass[11pt]{article}
\usepackage{amssymb}
\usepackage{epsfig,amsfonts,amssymb,amsthm}
\usepackage{amsmath}
\theoremstyle{plain}
\newtheorem{theorem}{Theorem}[section]
\newtheorem{prop}[theorem]{Proposition}

\newtheorem{lemma}[theorem]{Lemma}

\theoremstyle{definition}

\theoremstyle{remark}
\newtheorem{rem}{Remark}[section]

\newcommand{\ba}{\begin{eqnarray}}
\newcommand{\be}{\begin{equation}}
\newcommand{\ea}{\end{eqnarray}}
\newcommand{\ee}{\end{equation}}
\newcommand{\benn}{\begin{equation*}}
\newcommand{\eenn}{\end{equation*}}

\def\ve{\varepsilon}

\numberwithin{equation}{section}

\title{Localised modes due to defects in high contrast periodic
media via two-scale homogenization}
\author{I.V.
Kamotski$^{1}$ and V. P. Smyshlyaev$^{1,2}$}

\begin{document}

\maketitle

\bigskip

{ \small \noindent
$^1$ Department of Mathematics, University College London, Gower Street, London WC1E 6BT, UK.}\\

{ \small \noindent $^2$ Corresponding author:  e-mail:
  v.smyshlyaev@ucl.ac.uk}

\begin{abstract}
The spectral problem for an infinite periodic medium perturbed by a compact defect is considered.  
For a high contrast small $\ve$-size periodicity and a finite size defect we consider the critical
$\ve^2$-scaling for the contrast. We employ 
(high contrast) two-scale homogenization for deriving asymptotically explicit limit equations
for the localised modes (exponentially decaying eigenfunctions) and associated
eigenvalues. Those are expressed in terms of  the eigenvalues and eigenfunctions of a perturbed version of a
two-scale limit operator introduced by V.V. Zhikov with an emergent explicit nonlinear 
dependence on the spectral 
parameter for the spectral problem at the macroscale. Using the method of asymptotic expansions
supplemented by a high contrast boundary layer analysis we establish the
existence of the actual eigenvalues near the eigenvalues of the limit operator, with 
``$ \,\ve$ square root'' 
error bounds. 
An example  for circular or 
spherical defects in a periodic medium with isotropic homogenized properties is given and 
displays explicit limit eigenvalues and eigenfunctions. Further results on improved error bounds 
for the eigenfunctions are discussed, by combining our results with those of M. Cherdantsev ({\it Mathematika. 
2009;55:29--57}) based on the technique of strong two-scale resolvent convergence and associated 
two-scale compactness properties.
\end{abstract}


\newpage
\section{Introduction}

We dedicate this work to memory of Professor V.V. Zhikov. Among many important contributions of 
V.V. Zhikov are development of two-scale homogenization techniques for high contrast spectral problems,  
e.g. \cite{Zh1,Zh2,ZhP13}, and of error estimates in homogenization theory, e.g. \cite{JKO,Zh05,ZhP13}. 
The present work, whose preliminary version was completed in 2006, relates to the both topics and was greatly influenced by Zhikov then 
and has had a continued influence by his ideas since. 

Studying the spectral properties of operators with periodic coefficients, with and without
defects, has received considerable attention in the mathematical literature, see
e.g.  review \cite{K}. Some additional recent interest is motivated by problems in physics
associated with photonic or phononic crystals and photonic crystal fibers, 
see e.g. \cite{Jannop}, \cite{Russel}.
A photonic crystal fiber (PCF) for example represents geometrically a periodic medium (whose 
physical properties vary across the fiber but not along it), with
the defect being its ``core'', which is a propagating channel or a waveguide: electromagnetic 
waves of certain frequencies (the band gap frequencies) fail to propagate in the surrounding 
periodic medium and
hence remain localised inside the PCF, which allows for them to propagate along the core 
for long distances with little loss \cite{Russel}. Mathematically, the problem reduces
to an appropriate spectral problem at the cross-section of the PCF, cf. Figure \ref{fig1}. This is that
of characterisation of localised modes or eigenstates (whenever such exist) in the band gaps in the
Floquet-Bloch spectrum for the Maxwell's operator in the surrounding periodic medium with a fixed 
``propagation constant'' (the wave vector along the fiber). The latter cross-sectional geometry
is a periodic medium perturbed by a finite size heterogeneity (domain $\Omega_2$ in
Fig. \ref{fig1}). The problem is hence first in detecting the band gaps in the periodic medium without
defects and then in finding, in the presence of a defect, the extra point spectrum in the gaps as
well as the associated eigenfunctions, the localised states. In the present work we aim at detecting 
such localised modes in an asymptotically explicit form due to defects in high contrast periodic 
medium using the tools of (high contrast) homogenization theory. In physical terms, this corresponds to a 
simplified model with scalar rather than Maxwell's equations and with in effect zero propagation 
constant. We expect that this nevertheless captures the essence of the underlying effects, making 
thereby the proposed methodology more transparent and avoiding at the same time additional technical 
complications. For an asymptotic analysis of a full three-dimensional Maxwell PCF (although with a 
moderate contrast) see \cite{CKS14}. 
In any case, we expect the problem and the methods we develop here to be of mathematical interest, 
in particular in part of obtaining error bounds for high-contrast homogenization problems in the 
presence of a boundary layer due to the defect. 

Considerable literature is devoted to problems from the above described general class. Apart from 
numerous computational approaches (e.g., \cite{Bird3}), 
most of the mathematical treatments have been qualitative, establishing the existence of the band
gaps, of the point spectrum in the gaps in the presence of defects, some bounds on the number of the 
eigenvalues in the gap, on the pattern of the (exponential) decay of the eigenmodes, etc, see 
\cite{K} and further
references therein. If however the problem contains one or more small parameters, e.g. {\it high contrast}, 
often in the presence of other small parameters, e.g. thickness of thin periodic (high contrast) structures,
the asymptotic methods become potentially applicable for more explicit answers to the above questions. 
In mathematical literature, various results on the existence and the asymptotic 
description of the band gaps in high contrast periodic media have been obtained on this way,  e.g. 
\cite{
FigKunKuchm2, HL}. 
In particular, methods of homogenization theory,
including those of high contrast homogenization, have proven to be particularly fruitful for problems
from the above general class, see e.g. \cite{Zh1, Zh2, 
Bouchitte, 
GW}.

Hempel and Lienau \cite{HL} studied the spectral problem for a matrix-inclusion high contrast
periodic medium and established  asymptotically explicit band gaps using min-max
variational methods. Zhikov \cite{Zh1, Zh2}  has independently used for this case techniques
of high-contrast homogenization of ``double-porosity'' type. Related periodic medium
has  periodicity cell size $\varepsilon$ and the contrast between
the ``inclusion'' and the ``matrix'' of order $\varepsilon^2$
($\varepsilon$ is small), which is equivalent to the scaling of
\cite{HL}.  As a result the spectrum converges in the
sense of Hausdorff to an explicitly described limit 
spectrum which contains gaps. Zhikov, using the techniques of
two-scale convergence \cite{Ng, Al}, has made significant further advances having
additionally described an associated
(two-scale) limit operator and clarified further the convergence of the
spectra in terms of the strong two-scale resolvent convergence, the associated
convergence of the spectral projectors and certain additional
compactness properties. Zhikov has also shown that at the ``macroscale'' the spectral 
problem displays an emergent explicit nonlinear dependence on the spectral 
parameter, see \eqref{eq:eq5}, \eqref{eq:eq6} below.

In this paper, using the (high contrast) homogenization theory methods, we show that if such a
rapidly oscillating high-contrast medium is perturbed by a compact defect 
of size of order one,
asymptotically explicit eigenvalues and eigenfunctions can emerge
in the gap when $\varepsilon\to 0$. The essential spectrum is
known to remain unchanged under such a perturbation \cite{Birm,
AA, FK1}, and the existence of the point spectrum in the gaps of
the unperturbed operator has also been established on some
accounts together with some estimates on the number of eigenvalues
in the gap, see \cite{AA, FK1} and further references in
review \cite{K}. We argue that the homogenization techniques allow to
substantially refine this information, providing an explicit
asymptotic description and tight bounds on the (convergent)
eigenvalues and eigenfunctions, and ultimately on their number via
some kind of ``asymptotic completeness'' of the spectrum in the
gap as described by an explicit limit operator.

We employ in this work the method of asymptotic expansions supplemented
by its rigorous justification, in the case of regular boundaries for both
the periodic inclusions and the defect. This allows to obtain not only an
explicit description of the limit equations and the convergence 
results, but also to establish {\it the rate} of convergence: the main technical result 
of this work is
the error estimate \eqref{mainresult}, with the ``$\ve$-square-root'' bound
being typical for classical homogenization with boundary-layer effects,
see e.g. \cite{JKO} and further references therein. The method of asymptotic
expansions in the ``moderate contrast'' classical homogenization as well as its
rigorous justification are well developed, see e.g. \cite{BLP, SanPal,
BaPan, JKO}. Applications of asymptotic
methods for high contrast 
homogenization can be
found e.g. in \cite{Sandr}. On the other hand, error estimates in homogenization
can be obtained by other methods, see e.g. recent review \cite{ZhP16}, including the so called spectral method, e.g.
\cite{Beliaev, Zh3, AConca, BirmSusl, Zh4}, as well as by modifications of the
asymptotic expansions method with no assumptions on the regularity of the
coefficients, e.g. \cite{Zh05}. One novel technical ingredient in the present
work is perhaps the execution of a delicate boundary layer
asymptotic analysis in the high contrast case 
(Section 5). 

Cherdantsev \cite{cherd} considered essentially the same problem  
via the alternative method of two-scale convergence, following the general ideas of Zhikov \cite{Zh1,Zh2}. 
With an additional restriction on the 
boundary-layer inclusions, he established not only strong two-scale resolvent convergence but also associated compactness properties via some novel technical analysis of an $\ve$-uniform exponential decay of the 
eigenfunctions in the gap. As a result he obtained stronger results on the ``full'' Hausdorff spectral convergence, although without the error estimates. Since our approach and that of \cite{cherd} are essentially 
independent and the results complement each other, it is natural to combine them as we also discuss in the present paper. 

The structure of the paper is the following.
We first formulate the problem (Section 2), then give an explicit description
of the limit problem in Section 3. Section 4 establishes the
convergence and error bounds for the eigenvalues and the
eigenfunctions (Theorem \ref{thm1}). 
We give full proof of
Theorem \ref{thm1}, whose most technical part (Theorem \ref{l:estim}, including a number
of accompanying technical lemmas and propositions) is given in the Section 6.
An explicit example illustrating the 
existence of the point spectrum for the limit operator for
spherical defects in a surrounding periodic medium with isotropic effective
properties is given in Section 6. Finally we discuss the improved error estimates on 
the eigenfunctions 
by combining our results with those of Cherdantsev \cite{cherd} in  
Section 7. Appendix A gives formal
derivation of the limit problem, as well as of the 
correctors required for the rigorous justification.

\section{Formulation of the problem}

We consider a high contrast two-phase periodic medium with a small periodicity size and with
a ``finite size'' defect filled by a third phase.
The geometric configuration is displayed on Figure \ref{fig1}.

\begin{figure}[ht]

\centerline{\epsfig{file=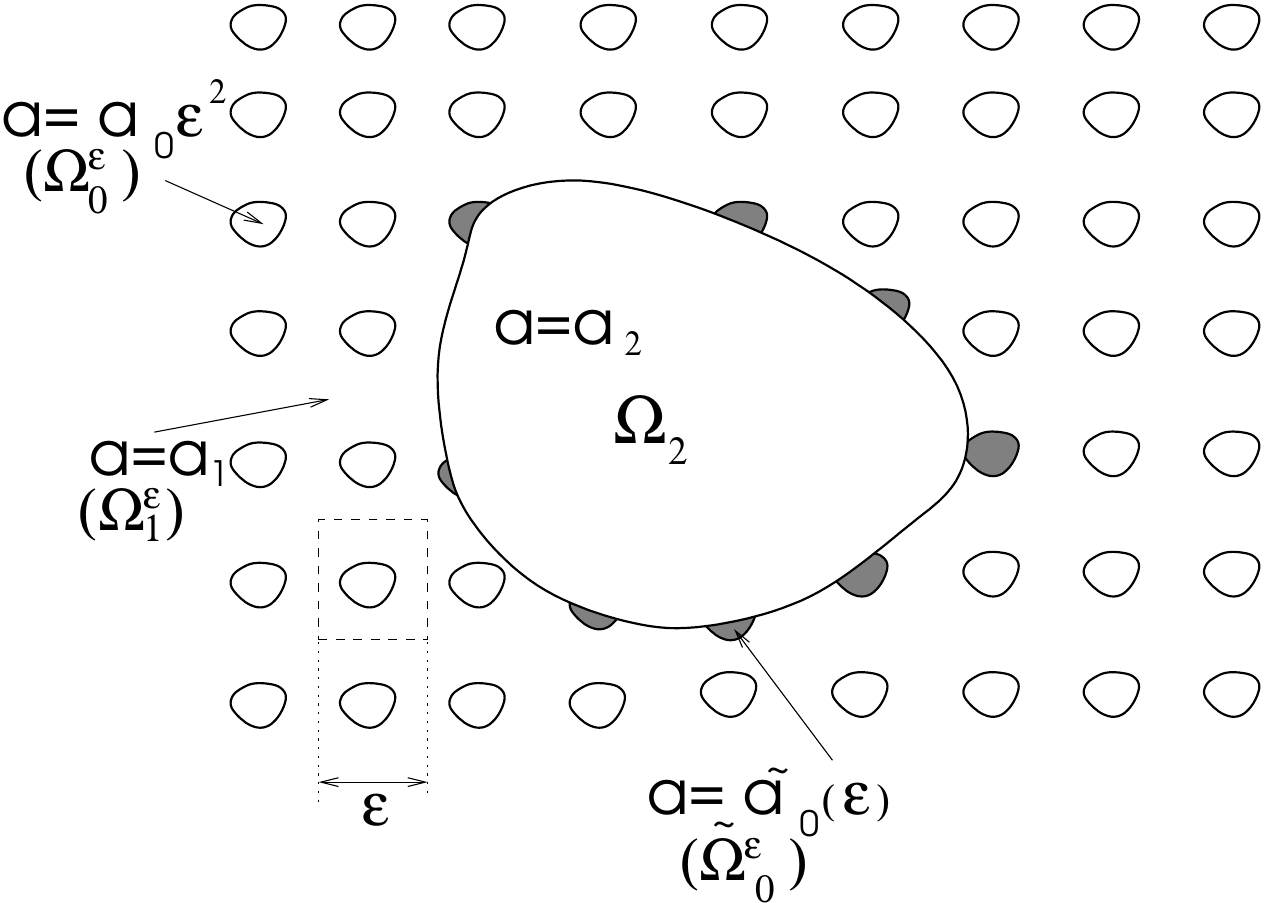}}
\caption{Geometric configuration: a defect in a rapidly oscillating high contrast medium} \label{fig1} \end{figure} 

The precise mathematical formulation is the following. Let $Q:=[0,1)^n$ be the reference
periodicity cell in $\mathbb{R}^n$, $n\geq 2$, and let $Q_0\subset Q$ be a 
domain (a
``reference inclusion'') in $Q$,  
with infinitely smooth boundary $\partial Q_0$, 
$\overline{Q_0}\subset Q$ (the overbar denotes the closure of the set). 
Denote by  $Q_1$ the complement of $Q_0$ in $Q$, 
$Q_1:=  Q\backslash \overline{Q_0}$.
 Let $\hat Q_0^\varepsilon$ be the corresponding
contracted set, i.e.  $\hat Q_0^\varepsilon := \{x:
x/\varepsilon \in Q_0\}$, where $\ve>0$ is a small positive parameter.
We denote by $Q_0^\varepsilon$ the
$\varepsilon$-periodic cloning of $\hat Q_0^\varepsilon$,
i.e. $Q_0^\varepsilon:=\hat Q_0^\varepsilon+\ve\mathbb Z^n$.
Let the ``defect domain'' $\Omega_2$  be an $\varepsilon$-independent bounded domain
with infinitely smooth boundary. We denote by
$\tilde Q_0^\ve$ the set of all the inclusions in $Q_0^\varepsilon$ which
 intersect with the boundary $\partial\Omega_2$ of $\Omega_2$, and by
$\tilde\Omega_0^\varepsilon$ the
union of all the parts from $\tilde Q_0^\varepsilon$ outside
 $\Omega_2$, {\it i.e.}
$\tilde\Omega_0^\varepsilon:=\tilde Q_0^\varepsilon\backslash\overline{{\Omega_2}}$,
see Figure \ref{fig1}.

One phase, the ``inclusions phase'' of the resulting composite medium,
denoted $\Omega_0^\varepsilon$, is the collection of all the small inclusions lying
entirely outside the defect $\Omega_2$, i.e.
 $\Omega_0^\varepsilon:=Q_0^\varepsilon
\backslash \left(\Omega_2\cup\tilde Q_0^\varepsilon\right)$.
The ``matrix phase'',
denoted $\Omega_1^\varepsilon$, is the complement to the inclusions outside
the defect, i.e.
$\Omega_1^\varepsilon:= \mathbb{R}^n\backslash \overline{\left(Q_0^\varepsilon
\cup \tilde Q_0^\varepsilon \cup \Omega_2\right)}$.

We assume that the matrix and the defect are filled with materials with
$\ve$-independent uniform properties $a_1$ and $a_2$ respectively, and the inclusions 
are filled with materials with $\ve$-dependent (uniform) properties: $a_0(\ve)$ and
$\tilde a_0(\ve)$ for the ``full'' (domain $\Omega_0^\varepsilon$) and the ``cut''
(domain $\tilde\Omega_0^\varepsilon$)) inclusions, respectively. Mathematically,
for every positive (small enough) $\ve$ we consider the spectral problem
\begin{equation}
   A_\varepsilon u^\varepsilon\,=\,\lambda(\varepsilon)
    u^\varepsilon
      \label{eq:eq1}
\end{equation}
for operator $A_\varepsilon$, self-adjoint in $L^2\left(\mathbb{R}^n\right)$ with a 
domain $D(A_\ve)$ dense in $H^1({\mathbb R}^n)$,
\begin{equation}
   A_\varepsilon u^\varepsilon\,:=\, -\nabla\cdot \biggl(a(x,\varepsilon) \nabla
u^\varepsilon(x)\biggr),\ \  x\in  \mathbb{R}^n,
      \label{eq:eq1new}
\end{equation}
where
\begin{equation}
a(x,\varepsilon)=
    \left\{
      \begin{array}{ll}
         a_0(\varepsilon), & x\in \Omega_0^\varepsilon,
          \\
         \tilde a_0(\varepsilon), & x\in \tilde\Omega_0^\varepsilon,
          \\
          a_1, & x\in \Omega_1^\varepsilon,
          \\
         a_2, & x\in \Omega_2.
      \end{array} \right.
      \label{eq:eq2}
\end{equation}

We assume that
\begin{equation}
   a_0(\varepsilon)=a_0\varepsilon^2
      \label{eq:eq3}
\end{equation}
(which is sometimes called a double porosity-type scaling), with $a_0$, as
well as $a_1$ and $a_2$ being arbitrary positive constants. There is
a degree of freedom in this work for selecting the scaling for
$\tilde a_0(\varepsilon)$ in the ''boundary layer'', so we only
require that
 \be
    0<\,\tilde
    a_0(\varepsilon)\,\leq\,\tilde A_0 \label{tildea0bounds}
 \ee
with some $\ve$-independent positive  $\tilde A_0$. Hence the
results of the present paper remain valid, for example, both for the
double porosity scaling where $\tilde a_0(\varepsilon)=\tilde
a_0\varepsilon^2$ with some $\ve$-independent positive $\tilde a_0$,
in particular for $\tilde a_0=a_0$ (physically, the ``cut''
inclusions outside the defect being kept), as well for $\tilde
a_0(\varepsilon)$ being ``of order one'', e.g. $\tilde
a_0(\varepsilon)=A_0$, in particular for $A_0=a_1$ or $A_0=a_2$ (the
cut inclusions being replaced by the matrix or defect material).
We also remark that the results presented in this paper will remain equally
valid for the boundary inclusions $\tilde Q_0^\ve$ also cutting
out of the ``defect'' part $\Omega_2$ (the maximal generality
has not been pursued to avoid unnecessary further technical complications).
Emphasize however that for subsequent stronger results on the
Hausdorff convergence of the spectra and the convergence of the
eigenfunctions (see \cite{cherd} and  Section 7 below) a more restrictive 
(essentially ``order one'')  
choice for $\tilde a_0(\varepsilon)$ becomes necessary, to exclude additional 
modes which might otherwise exist near the defect's boundary.

The equation (\ref{eq:eq1}) is understood in the usual weak sense,
implying the continuity of $u^\varepsilon$ and of the conormal
derivatives at the boundaries of $\Omega_0^\varepsilon,
\Omega_1^\varepsilon, \Omega_2$ and $\tilde\Omega_0^\varepsilon$.

Hence for every fixed positive $\ve$ (\ref{eq:eq1})--(\ref{eq:eq1new})
represents a spectral problem for an operator with periodic
coefficients ``perturbed'' by a localised defect.
We are mostly interested in the existence and asymptotics for the
eigenvalues $\lambda(\ve)$ and the associated localised solutions
(eigenfunctions) $u^\ve(x)$ of (\ref{eq:eq1}) when $\varepsilon \to 0$.

\section{Homogenization and the limit problem}

We describe in this section the formal asymptotic procedure for solving
(\ref{eq:eq1})--(\ref{eq:eq1new}) when $\ve\to 0$. It is rigorously
justified in the subsequent sections.

One can seek a formal solution to the spectral problem
(\ref{eq:eq1})--(\ref{eq:eq1new}) in the form of a standard two-scale ansatz:
\begin{eqnarray}
u^\varepsilon(x)&=&u^{(0)}(x,x/\varepsilon)+\varepsilon
u^{(1)}(x,x/\varepsilon)+\varepsilon^2
u^{(2)}(x,x/\varepsilon)\, +\,r^\varepsilon(x),
\label{ansatz1}\\
\lambda(\varepsilon)&=&\lambda_0+o(1),
\label{ansatz2}
\end{eqnarray}
where $u^{(0)}(x,y),$  $u^{(1)}(x,y)$ and $u^{(2)}(x,y)$ are functions to be
determined which are $Q$-periodic in $y$, the remainders $r^\varepsilon(x)$
and $o(1)$ are expected to be small when $\varepsilon \to 0$, with $\lambda_0$
and $u^{(0)}(x,y)$ subsequently having the meaning of the eigenvalues
and eigenfunctions of a ``limit problem''. [We subsequently show that the 
remainder $o(1)$ in \eqref{ansatz2} is in fact ``of order $\ve^{1/2}$'', i.e. 
$O(\ve^{1/2})$, see \eqref{mainresult}.]

A formal substitution of (\ref{ansatz1})-(\ref{ansatz2}) into
(\ref{eq:eq1}) results upon straightforward calculation in the
following structure of the main-order term $u^{(0)}(x,y)$, see
Appendix A:
\begin{equation}
u^{(0)}(x,y)=
    \left\{
      \begin{array}{ll}
         u_0(x), & x\in \Omega_2\,\,\mbox{ or }\,\, x\in\mathbb{R}^n \backslash \Omega_2, \, y\notin Q_0 ,
          \\
          u_0(x)+v(x,y), & x\in
\mathbb{R}^n \backslash \Omega_2, \, y\in Q_0,
      \end{array} \right.
      \label{u0}
\end{equation}
which highlights the fact that $u^{(0)}$ varies only at the ``slow''
scale $x$ (as is the case in the classical, i.e. not
high-contrast, homogenization) everywhere outside the domain of
``soft'' inclusions $\Omega_0^\varepsilon$, however may depend on
the fast variable $y=x/\varepsilon$ in $\Omega_0^\varepsilon$.
Further, the pair of functions $(u_0(x), v(x,y))$ must solve the
following limit coupled  spectral problem:
\begin{eqnarray}
-\nabla\cdot a_2 \nabla u_0(x)&=&\lambda_0
    u_0(x),\ \  x\in  \Omega_2,
\label{330}\\
-\nabla\cdot A^{\rm hom} \nabla u_0(x)&=& \lambda_0\left(u_0+\langle
v\rangle_y \right),\ \ x\in  \mathbb{R}^n\backslash \Omega_2,
\label{331}\\ 
-\,a_0\Delta_yv &=& \lambda_0\left(u_0\,+\,v\right), \ \ y\in
Q_0; \,\, v=0,\,y\in \partial Q_0 \,\, \left( x\in
\mathbb{R}^n\backslash \Omega_2\right), \label{332}\\
\left(u_0\right)_-&=&\left(u_0\right)_+, \,\, a_2\left(\frac{\partial
u_0}{ \partial n}\right)_-= \left(A_{ij}^{\rm hom}\frac{\partial
u_0}{\partial x_j}n_i\right)_+, \, x\in\,\partial \Omega_2. 
\label{limsystem}
\end{eqnarray}
Here
\begin{equation}
\langle v\rangle_y(x)\,:=\,|Q|^{-1}\int_{Q}v(x,y)\, dy
\label{yaverage} \end{equation} denotes the averaging with respect
to $y$ over the periodicity cell $Q$ (extending $v$ by zero outside
$Q_0$); 
 $\Delta_y$ is the Laplace
operator, 
$(\cdot)_-$ and $(\cdot)_+$ denote respectively the interior
and exterior limit values  of the appropriate entities at the
boundary $\partial \Omega_2$ of $\Omega_2$, $n$ is the interior unit
normal to $\partial \Omega_2$, 
$\partial/\partial n:=n\cdot\nabla$ is the normal derivative, 
summation is henceforth implied with
respect to repeated indices. In (\ref{331})
$A^{\rm hom}=\left(A^{\rm hom}_{ij}\right)$ is the standard 
``soft inclusions'' (or perforated domain)
homogenized matrix  for the above described periodic medium with
$a_0=0$, see e.g. \cite{JKO} \S3.1:
\begin{equation}
A^{\rm hom}_{ij}\xi_i\xi_j=\inf_{w\in C^\infty_{\rm per}(Q)}
\int_{Q \backslash Q_0}a_1\left\vert\xi+\nabla w\right\vert^2\,dy\,\,\,\,\left(\xi \in
\mathbb R^n\right), 
\label{eq:ahom}
\end{equation}
where $C^\infty_{per}(Q)$ denotes infinitely smooth $Q$-periodic functions. 

Notice in passing that, following the pattern of Zhikov \cite{Zh2}, the above limit problem 
(\ref{330})--(\ref{limsystem}) can be interpreted\footnote{See some further discussion in
Section 7 on 
the
limit operator.} as a
spectral problem for a two-scale limit non-negative self-adjoint operator (which we will denote
$A_0$) acting in the following Hilbert space ${\cal H}_0$:
\begin{eqnarray}
{\cal H}_0=\biggl\{ u(x,y)\in L^2\left(\mathbb{R}^n\times Q\right)
&\biggl\vert& u(x,y)=u_0(x)+v(x,y), \, \ \ u_0\in
L^2\left(\mathbb{R}^n\right),  \biggr.\biggr.
\nonumber\\
&&\biggl.
v\in\,L^2\biggl(\mathbb{R}^n\times Q 
\biggr)
;\, 
v(x,y)=0 \mbox{ if } y\in Q_1 \mbox{ or } x\in\Omega_2 \biggr\}. 
\nonumber\\ 
\label{Hspace}
\end{eqnarray}
The operator $A_0$ is  generated by the (closed) symmetric and
non-negative quadratic form $B_0$ defined on dense in ${\cal H}_0$ domain 
\begin{eqnarray}
{\cal V}=\biggl\{ u(x,y)\in {\cal H}_0 
&\biggl\vert& u(x,y)=u_0(x)+v(x,y), \, \ \  u_0\in
H^1\left(\mathbb{R}^n\right),  \biggr.\biggr.
\nonumber\\
&&\biggl.
v\in\,L^2\biggl(\mathbb{R}^n; H^1_{per}(Q)
\biggr)
;\, 
v(x,y)=0 \mbox{ if } y\in Q_1 \mbox{ or } x\in\Omega_2 \biggr\}, 
\nonumber\\ 
 \label{Qdomain}
\end{eqnarray}
where $ H^1_{per}(Q)$ stands for the closure in  $H^1(Q)$ of $C_{per}^\infty(Q)$. 
The limit two-scale form $B_0$ is then 
defined as follows: for $u(x,y)=u_0(x)+v(x,y)\in\,{\cal V}$ and $w(x,y)=w_0(x)+z(x,y)\in\,{\cal V}$,
\begin{equation}
B_0(u,w)\,=\,\int_{\Omega_2}a_2\nabla u_0\cdot\nabla w_0 \,dx+
\int_{\mathbb{R}^n\backslash\Omega_2}A^{\rm hom}\nabla u_0\cdot\nabla
w_0\,dx\,+\, \int_{\mathbb{R}^n\backslash\Omega_2}\int_{Q_0}
a_0\nabla_y v\cdot\nabla_y z\, dy\,dx. \label{Qform}
\end{equation}
The resulting (self-adjoint) operator $A_0$ is defined in standard way
on a dense domain $D(A_0) \subset {\cal V}$. 
The limit spectral problem (\ref{330})--(\ref{limsystem}) can then 
be equivalently stated as follows: find $u=u_0+v\in {\cal V}$, $u\neq 0$, and $\lambda_0\geq 0$ such that 
\[
B_0(u,w)\,=\,\lambda_0 (u,w)_{{\cal H}_0}, \,\,\, \mbox{for all } w\in {\cal V}, 
\]
where $(u,w)_{{\cal H}_0}$ is the inner product in ${\cal H}_0$ taken as the standard inner product in 
$L^2({\mathbb R}^n\times Q)$. 
(The latter construction of $A_0$ is not of an extensive 
use in the present paper and hence is not elaborated upon here.)

On the other hand, the limit problem (\ref{330})--(\ref{limsystem})
can be re-arranged as follows, cf. \cite{Zh1, Zh2}. 
Assume 
$\lambda_0\neq \lambda_j$ for all $j\geq 1$, and let $\lambda_j$ and 
$\varphi_j(y)$ be the eigenvalues and the
(normalized) eigenfunctions respectively of $-a_0\Delta_y$ 
 in $Q_0$ with Dirichlet boundary conditions
on $\partial Q_0$. 
Boundary value problem \eqref{332} 
 implies that $v(x,y)$ can be uniquely presented as
 \be
    v(x,y)=u_0(x)V(y),
\label{Vy}
 \ee
where $V$ is a solution of the problem
 \be
    -a_0\triangle_y V=\lambda_0 V+\lambda_0,\  y\in Q_0;
    \  V=0,\ y\in \partial Q_0.
 \label{eq:V}
 \ee
It is further assumed that $V$ is extended by zero to $Q$ and is then
periodically extended on the whole $\mathbb{R}^n$. 
Then, substituting (\ref{Vy}) back into (\ref{331}) we arrive
at the following spectral problem for $u_0$ with a nonlinear dependence on the
spectral parameter $\lambda$:
\begin{equation}
    -\nabla\cdot a_2 \nabla u_0(x)=\lambda_0
    u_0(x),\ \  x\in  \Omega_2,
      \label{eq:eq4}
\end{equation}
\begin{equation}
    -\nabla\cdot A^{\rm hom} \nabla u_0(x)=\beta(\lambda_0)
    u_0(x),\ \  x\in  \mathbb{R}^n\backslash \Omega_2, 
      \label{eq:eq5}
\end{equation}
\begin{equation}
\left(u_0\right)_-\,=\,\left(u_0\right)_+, \,\, a_2\left(\frac{\partial
u_0}{ \partial n}\right)_-= \left(A_{ij}^{\rm hom}\frac{\partial
u_0}{\partial x_j}n_i\right)_+, \, x\in\,\partial \Omega_2.
\label{316}
\end{equation}
Here
\begin{equation}
   \beta(\lambda):=\,\lambda\,+\,\lambda\langle V\rangle
	\label{betadef}
	\ee
is the function introduced by Zhikov . 
Applying the spectral decomposition to (\ref{eq:V}),  cf. \cite{Zh1, Zh2}, 
\begin{equation}
V(y)\,=\,\lambda_0
\,\sum_{j=1}^\infty\frac{\langle\varphi_j\rangle_y}
{\lambda_j-\lambda_0}\varphi_j(y), \label{specdecomv}
\end{equation}
and substituting the latter into \eqref{betadef} results in 
\be 
\beta(\lambda)\,=\,\lambda\,+\,
\lambda^2\sum_{j:\,
\langle\varphi_j\rangle_y\neq 0}
\frac{\langle\varphi_j\rangle_y^2}{\lambda_j-\lambda},
      \label{eq:eq6}
\end{equation}
which is a sign-changing function, also derived in \cite{Zh1, Zh2}, singular near the inclusions' eigenvalues $\lambda_j$, 
$j=1,2,...$, 
 see Figure 2. 
(Strictly speaking, $\lambda_1$, $\lambda_2$, $\lambda_3$, ... on Fig. 2 
denote only  the eigenvalues corresponding to $\varphi_j$ with 
non-zero mean, ordered by increasing values.) 
\begin{figure}[ht]
\centerline{\epsfig{file=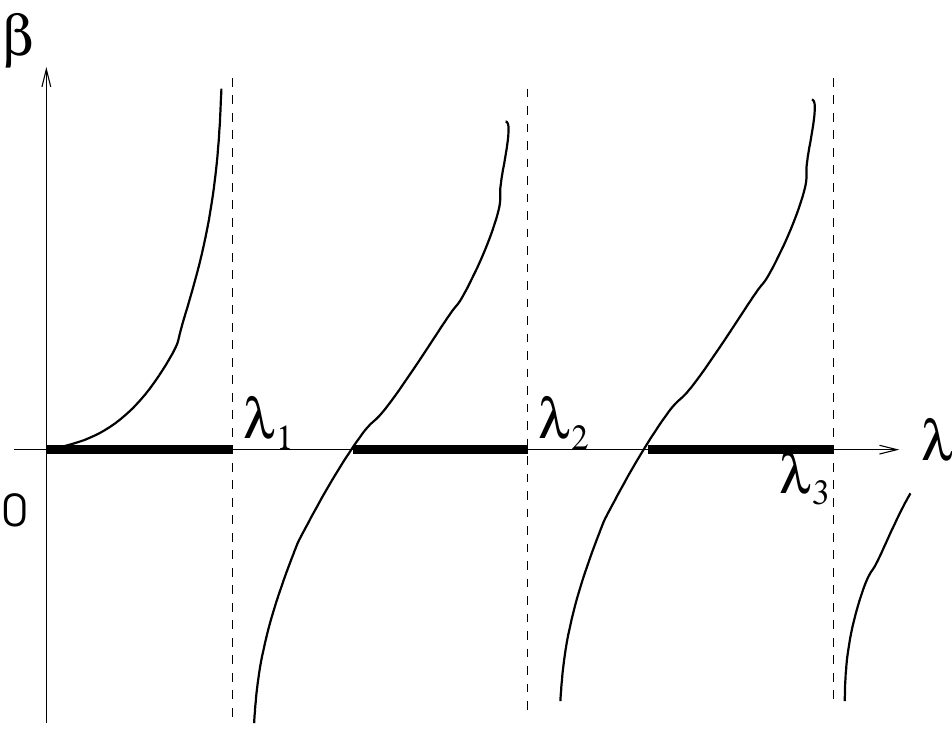}} \caption{Function
$\beta(\lambda)$} \label{fig2}
\end{figure}

The spectrum of the ``unperturbed'' limit operator (i.e. with no defects)
consists of the union of $\lambda$ where $\beta(\lambda)\geq 0$ and
of $\lambda_j$, $j=1,2,...$, \cite{Zh2}. If
$\langle\varphi_j\rangle_y=0$ the spectrum of the limit operator
also contains (infinite multiplicity) point spectrum at
$\lambda=\lambda_j$. Hence the limit operator has gaps when
$\beta(\lambda)<0$, $\lambda\neq \lambda_j, j=1,2,...$ .

In this work, for the perturbed linear operator $A_0$, restricting ourselves
only to values of $\lambda$ in the gaps of the
unperturbed limit operator, we define $\lambda_0$ to be an
eigenvalue of $A_0$ if $\beta(\lambda_0)<0$,
$\lambda_0\neq \lambda_j$, $j\geq 1$, and the system
(\ref{eq:eq4})--(\ref{eq:eq6}) admits a solution $u_0(x)$
decaying at infinity (hence, decaying exponentially since
$\beta(\lambda_0)<0$). We define as an eigenfunction (or an
eigenvector) of $A_0$ the pair of functions $(u_0,v)$, where
$u_0(x)$ is the above solution and $v(x,y)$ is related to
$u_0(x)$ via (\ref{332}) (equivalently via \eqref{Vy} and \eqref{eq:V}, (\ref{specdecomv}))
for $y\in Q_0$ and $x\in \mathbb{R}^n\backslash \Omega_2$
and extended by zero for the remaining values of $x$ and
$y\in Q$.

One can show from (\ref{eq:eq4})--(\ref{316})
 that the perturbed limit operator $A_0$ inside the
gaps of the unperturbed operator can only develop isolated
eigenvalues of finite multiplicity, see \cite{cherd} \S7.
As we show in Section 6, the problem (\ref{eq:eq4})--\eqref{316} admits in
some cases an explicit calculation of its eigenvalues and
eigenfunctions.

\section{Convergence and error bounds for eigenvalues and eigenmodes}

The main result of the present work is the following theorem
establishing the closeness of the spectrum of the
original operator $A_\varepsilon$ to  the spectrum of the above described
limit operator $A_0$ as $\varepsilon\to 0$:
\begin{theorem}\label{t:main}
\label{thm1} Let $\lambda_0$ be an eigenvalue of limit operator $A_0$,
 $\beta(\lambda_0)<0$,
$\lambda_0\neq\lambda_j, j\geq 1$. Then there exists $\ve_0>0$ and a
constant $C_1>0$ independent of $\varepsilon$  such that
for any $\,0<\ve\leq \ve_0$ there exists an isolated
eigenvalue $\lambda(\varepsilon)$ of operator $A_\varepsilon$ of finite
multiplicity, such that
\be
\left|\lambda(\varepsilon)-\lambda_0\right|\,\leq\,C_1\varepsilon^{1/2}.
\label{mainresult}
\end{equation}

Moreover if $(u_0,v)$ is an eigenfunction of $A_0$ which corresponds to $\lambda_0$ then
the function \begin{equation} u^{\rm appr}(x,\varepsilon):=
    \left\{
      \begin{array}{ll}
         u_0(x)+v(x,x/\varepsilon), & x\in \Omega_0^\varepsilon,
          \\
          u_0(x), & x\in \Omega_1^\varepsilon \cup \Omega_2 \cup
\tilde\Omega_0^\varepsilon,
      \end{array} \right.
      \label{eq:eq27}
\end{equation}
is an approximate eigenfunction for $A_\varepsilon$ at least in the
following sense\footnote{See \cite{cherd} and the discussion in Section 7 on
strengthening, under an additional restriction, of this result on convergence of the eigenfunctions.}:
there exist constants $c_j(\varepsilon)$ such that
\begin{equation}
\left\| u^{\rm appr}-\sum_{j\in J_\varepsilon}
c_j(\varepsilon)u_j^\varepsilon\right\|_{L_2(\mathbb
R^n)}<C_2\varepsilon^{1/2}, \label{eq28}
\end{equation}  where
$J_\varepsilon=\{j:|\lambda^{(j)}(\varepsilon)-\lambda_0|<C\varepsilon^{1/2}\}$ 
 is a finite set of indices (for each $\ve$), 
and $\lambda^{(j)}(\varepsilon)$, $u^\varepsilon_j(x)$ are eigenvalues
and $L_2$-normalized eigenfunctions of $A_\varepsilon$, and the
constants $C_2$ and $C$ are independent of $\varepsilon$.
\end{theorem}

{\bf Proof:} We first  establish estimates
somewhat related to the ``strong resolvent'' convergence of operators
$A_\varepsilon$ when $\varepsilon\to 0$. Strictly speaking, those are
related to some generalization of the latter:  the usual resolvent
convergence is not suitable for our purposes because the space
where the limit operator $A_0$ acts differs from the
space natural for operators $A_\varepsilon$. One therefore has to
refer to the so-called {\it two-scale} strong resolvent convergence,
see Zhikov \cite{Zh1, Zh2}.

Let $u^0(x,y)=\left(u_0(x),v(x,y)\right)\in D(A_0)$ be an eigenfunction
of the operator $A_0$ corresponding to an eigenvalue $\lambda_0$,
$\beta(\lambda_0)<0$, $\lambda_0\neq\lambda_j$, $j\geq 1$. 
Denote by $U_\varepsilon$ the ``transfer'' operator,
constructing from $u^0$ the approximate eigenfunction $u^{\rm appr}$
via (\ref{eq:eq27}), i.e. 
\begin{equation} (U_\varepsilon u^0)(x):=
    \left\{
      \begin{array}{ll}
         u_0(x)+v(x,x/\varepsilon), & x\in \Omega_0^\varepsilon,
          \\
          u_0(x), & x\in \Omega_1^\varepsilon\bigcup\Omega_2\cup
\tilde\Omega_0^\varepsilon.
      \end{array} \right.
\label{eq:uapr}
      \end{equation}
\noindent Notice that due to the regularity of $u_0$ and $v$ (which
solve (\ref{eq:eq4})--(\ref{316})), decomposition (\ref{specdecomv})
and the exponential decay of $u_0$ at infinity when
$\beta(\lambda_0)<0$, $U_\varepsilon u^0\in L_2(\mathbb R^n)$.

Denoting by $I$ the unity operator, 
we formulate next the main technical statement of this work, close
to that of the two-scale resolvent convergence, cf. \cite{Zh1, Zh2}:

\begin{theorem} \label{l:estim}
Let $u^0(x,y)=\left(u_0(x),v(x,y)\right)\in D(A_0)$ be an eigenfunction
of the operator $A_0$ :
\be A_0 u^0=\lambda_0 u^0,\,\,\,\,\lambda_0\neq \lambda_j,\,\, j\geq 1;\,\, \beta(\lambda_0)<0.
\ee
Consider the following (resolvent) problem for the original operator $A_\ve$:
 \be (A_\varepsilon+I)
\tilde u^\varepsilon=(\lambda_0+1)U_\varepsilon u^0.
\label{problem42}
\ee
Then  \be\|U_\varepsilon u^0-
\tilde u^\varepsilon\|_{L_2(\mathbb{R}^n)}\leq C\varepsilon^{1/2},
\label{eq:main}
\ee
with a constant $C$ independent of $\varepsilon$.
\end{theorem}

\begin{rem} The above estimate can be  equivalently rewritten in a form which
somewhat clarifies the role of operator $U_\varepsilon$:
 \be\|U_\varepsilon(A_0+I)^{-1}u^0-(A_\varepsilon+I)^{-1}U_\varepsilon
 u^0
\|\leq C \varepsilon^{1/2},
\label{eq:isom}\ee
where $u^0$ is
an eigenfunction of operator $A_0$, cf. \cite{Zh05}.
\end{rem}
The proof of Theorem \ref{l:estim} constitutes the key technical
component for establishing the main result (\ref{mainresult}). 
The proof of the central error bound (\ref{eq:main})
 requires, among other ingredients, the development of a high
contrast 
version of the asymptotic analysis of the 
boundary layer near the boundary of the defect $\Omega_2$, conceptually 
somewhat similar to e.g. \cite{JKO} \S 1.4. The complete proof of
Theorem  \ref{l:estim} is given in Section 5. We remark here that the 
need of executing a series of technical error estimates in Section 5  
is caused by the fact that ``globally'' we can explicitly construct only 
the main order term in the asymptotic expansion, with the major obstacle 
being the need to control the effect of the boundary layer near the defect's 
border $\partial\Omega_2$. Constructing or analyzing the boundary layer in homogenization 
is still an open problem in general, even in the classical (``moderate contrast'') case, 
see e.g. \cite{NeussRadu} 
for some 
developments. 
Nevertheless, the boundary layer's effect 
can instead be somewhat controlled via an order-optimal error bound of order $\ve^{1/2}$ in the $H^1$ norm, 
see e.g. \cite{JKO} \S 1.4 for the case of boundaries with Dirichlet or Neumann conditions. 
In the present work we face however the {\it high contrast} 
version of the boundary layer problem and in effect show that 
even in this  case (for the interface conditions) the boundary layer accounts for the
 ``$\ve$-square root'' error, but in the $L_2$ norm, see \eqref{eq:main}. 
The series of technical lemmas and propositions in Section 5 ensure 
that the approximation $U_\ve u^0$ is close to the exact solution 
$\tilde u^\ve$ in the sense of appropriate quadratic forms, see \eqref{eq:form}. 
Namely, the main-order contributions of the quadratic forms corresponding to 
$U_\ve u^0$ and $\tilde u^\ve$  coincide and the errors are 
controllably small, including those due to the boundary layer. 

Denote by $\sigma_{\rm ess} (A_\varepsilon)$ the essential spectrum of
operator $A_\ve$.
The next step is, partly, a specialization of a more general methodology, see
e.g. \cite{JKO} \S 11.1, to the present context:

\begin{lemma} \label{l:general}
Let $\lambda_0$ be an eigenvalue of operator $A_0$, $\beta(\lambda_0)<0$,
$\lambda_0\neq \lambda_j$, $j\geq 1$, and let $u^0=(u_0,v)$ be associated
eigenfunction.
 Then:
\begin{enumerate}
\item[(i)]
For sufficiently small $\varepsilon$ there exists $c>0$ independent of
$\varepsilon$, such that 
\begin{equation}
(\lambda_0 -c,\lambda_0
+c)\bigcap \sigma_{\rm ess} (A_\varepsilon)=\,\emptyset.
\label{noess}
\end{equation}
\item[(ii)]
There exists, for sufficiently small $\varepsilon$, an isolated eigenvalue
(hence of finite multiplicity) $\lambda(\varepsilon)$ of operator
$A_\varepsilon$, such that
\be|\lambda(\varepsilon)-\lambda_0|<C_1\varepsilon^{1/2},
\label{eq:geneigenestim} \ee with constant $C_1$ independent of
$\varepsilon$.
\item[(iii)]
There exist constants
$c_j(\varepsilon)$ such that \begin{equation} \|U_\varepsilon
u^0-\sum_{j\in J_\varepsilon}
c_j(\varepsilon)u_j^\varepsilon\|_{L_2}<C_2\varepsilon^{1/2},
\label{eq:genestim}
\end{equation}
 where
$J_\varepsilon=\{j:|\lambda^{(j)}(\varepsilon)-\lambda_0|<C\varepsilon^{1/2}\}$; 
$\lambda^{(j)}(\varepsilon), u^\varepsilon_j(x)$ are eigenvalues and 
($L_2$-normalized) 
eigenfunctions of $A_\varepsilon$, and the constants $C_1$ and $C_2$
are independent of $\varepsilon$.
\end{enumerate}
\end{lemma}
\begin{proof}

{\it (i)}
The assertion (\ref{noess}) follows from the Hausdorff convergence of the
spectra of the ``unperturbed'' (i.e. with no defects) operators with
periodic coefficients \cite{HL, Zh2}, and the stability of the essential
spectrum due to localized defects, e.g. \cite{FK1}.

{\it (ii)} To prove (\ref{eq:geneigenestim}) recall that the
distance from a point $\mu$ to the spectrum of a linear self-adjoint operator $B$
in $L_2(\mathbb R^n)$ can be bounded from above as follows 
\begin{equation}
\text{dist}(\mu,\sigma (B))\leq \frac{\|Bu-\mu
u\|_{L_2(\mathbb R^n)}}{\|u\|_{L_2(\mathbb R^n)}},  \label{eq:dist}
\end{equation}
with any $u\in D(B)$, $u\neq 0$. Let us take as $B$ and $\mu$,
respectively, $B=B_\varepsilon:=(A_\varepsilon+I)^{-1}$ and
$\mu=(\lambda_0+1)^{-1}$. Then $B$ is bounded with $D(B)=L^2({\mathbb R}^n)$, and 
obviously $\lambda\in
\sigma(A_\varepsilon)$ if and only if $(\lambda+1)^{-1}\in \sigma
(B_\varepsilon)$. Now select $u=U_\varepsilon u^0$. Then according to Theorem
\ref{l:estim}, see \eqref{eq:isom}, the numerator in  \eqref{eq:dist} can be
estimated as follows
$$\|(A_\varepsilon+I)^{-1}U_\varepsilon
u^0-(\lambda_0+1)^{-1}U_\varepsilon u^0\|_{L_2(\mathbb R^n)}=$$
 \be
=\|(A_\varepsilon+I)^{-1}U_\varepsilon
u^0-U_\varepsilon(A_0+I)^{-1} u^0\|_{L_2(\mathbb R^n)}\leq
C\varepsilon^{1/2},
\label{EFerror}
\ee
(where we have also used that $(1+\lambda_0)^{-1}u^0=(A_0+I)^{-1}u^0$).
Obviously, the denominator in (\ref{eq:dist}) is bounded from below
(e.g. $\Vert U_\ve u^0\Vert_{L_2(\mathbb R^n)}\geq
\Vert u_0\Vert_{L_2(\Omega_2)}>0$).
As a result, $\text{dist}(\mu,\sigma (B_\varepsilon))\leq c_1
\varepsilon^{1/2}$, with some $\ve$-independent $c_1$.
Using then, for example, for small enough $\ve$ obvious inequality, 
$\text{dist}(\lambda_0,\sigma (A_\ve))\leq L(\ve) \text{dist}(\mu,\sigma (B))$
with $L(\ve)=\left(\mu-c_1\ve^{1/2}\right)^{-2}$ we arrive at
$\text{dist}(\lambda_0,\sigma (A_\varepsilon))\leq c
\varepsilon^{1/2}$, with some $\ve$-independent positive constant $c$.
Finally notice that, for sufficiently small $\ve$, the interval $(\lambda_0
-c\varepsilon^{1/2},\lambda_0 +c\varepsilon^{1/2})$ may contain only
isolated eigenvalues of $A_\varepsilon$ by (\ref{noess}).
This proves the existence of eigenvalues $\lambda(\varepsilon)$ satisfying
\eqref{eq:geneigenestim}.

{\it (iii)} The assertion \eqref{eq:genestim} is a consequence of
(\ref{EFerror}) and general results, see e.g. \cite{VL} or
\cite{JKO} \S 11.1, and follows by applying spectral decomposition
of $U_\ve u^0$ with respect to $A_\ve$, and using $(i)$ and $(ii)$. 
Namely, if $P_\lambda^\varepsilon$ are spectral projectors of $A_\ve$, then 
\[
(A_\varepsilon+I)^{-1}U_\varepsilon
u^0-(\lambda_0+1)^{-1}U_\varepsilon u^0\,=\,
\int_0^\infty \left(\frac{1}{\lambda+1}\,-\,\frac{1}{\lambda_0+1}\right)\,dP^\ve_\lambda(U_\ve u^0), 
\]
and \eqref{eq:genestim} follows from \eqref{eq:main}, the orthogonality properties of the spectral projectors, 
and \eqref{eq:geneigenestim}. 
\end{proof}

Now the Theorem \ref{t:main} directly follows from Theorem \ref{l:estim} and
Lemma \ref{l:general}. $\square$

\begin{rem}
It is relatively straightforward to slightly modify the statement of 
Theorem \ref{t:main} for limit eigenvalues $\lambda_0$ with ``multiplicities''. 
Namely, let 
for a given $\lambda_0$, $\beta(\lambda_0)<0$, $\lambda_0\neq \lambda_j$, $j=1,2,...$, 
there exist $m$, $2\leq m<\infty$, linearly independent eigenfunctions 
$u^{(0)}_j(x,y)$, $j=1,...,m$ of the two-scale limit operator $A_0$. 
Then we claim that, for sufficiently small $\ve$, there exist {\it not less than 
$m$} eigenvalues of $A_\ve$ (counted with their own multiplicities) such that the 
error bound \eqref{eq:main} is valid for each of them. The above proof could be 
modified for this case by for example first ``orthogonalizing'' $u^{(0)}_j(x,y)$, 
$j=1,...,m$   (with respect to the inner product 
induced by the quadratic form $B_0$, see \eqref{Qform}), and then showing that the 
associated approximations $U_\ve u^0_j$, $j=1,...,m$ are ``approximately'' mutually orthogonal 
for small $\ve$ in the ``original'' space $L_2\left({\mathbb R}^n\right)$. The latter 
would then allow to modify the above existence and error bounds argument for the case of 
 multiplicities. We do not elaborate on this in detail to avoid further technical 
complications, but also since 
(under an additional restriction on the boundary inclusions) 
a further strengthening of this result is 
possible (via a further advance of the theory, see \cite{cherd}), to the effect that there are not only 
``at least'' but also ``at most'' as many eigenvalues as $m$ near $\lambda_0$: see \cite{cherd} and the discussion in Section 7. 
\end{rem}

\section{Proof of Theorem \ref{l:estim}}
\setcounter{equation}{0}

We give in this section a full proof of the main technical Theorem
\ref{l:estim}, whose key error bound \eqref{eq:main} establishes
the closeness of the exact solution $\tilde u^\ve$ of
\eqref{problem42} to the approximate solution $U_\ve u^0$
constructed via \eqref{eq:uapr} in terms of the solution
$u^0:=(u_0(x), v(x,y))$ of the limit problem
\eqref{330}--\eqref{limsystem}.

The proof of the Theorem \ref{l:estim} will be divided into a number of stages. 
The plan is roughly as follows. The closeness of $U_\ve u^0$ and $\tilde u$ 
is established employing the associated quadratic form $b_\ve$, see 
\eqref{eq:form} below, where $U_\ve u^0$ is replaced by its modification $U^\ve_1(x)$,  
incorporating some higher-order terms in the asymptotic expansion \eqref{ansatz1}, 
see \eqref{eq:u3}. Namely, we show that it is sufficient for our purposes to establish 
the closeness in the sense of \eqref{lemmatech1}, Lemma \ref{lemmatech}. A 
technical proof of Lemma \ref{lemmatech} itself then follows by first splitting the 
quadratic form in the left hand side of \eqref{lemmatech1} into those corresponding 
to $U^\ve_1(x)$ and $\tilde u$, and then splitting the former further into 
a number of components (corresponding to the various domains of the integration in \eqref{eq:form})  
and 
examining those separately, see Propositions \ref{proptildea0}--\ref{propa2}. For 
each component the main-order parts are explicitly evaluated and the errors  
are bounded. Eventually everything is assembled together and the main-order terms 
cancel each other as anticipated, whereas the errors are shown to be ``at worst'' of 
order $\ve^{1/2}$. (The latter $\ve^{1/2}$-errors correspond in a sense to the effect 
of the boundary-layer near the defect's border, and those of order 
$\ve$ or higher to the truncation of the asymptotic ansatz away from it.) An essential specific 
technical ingredient used in the course of implementing the above strategy is the 
employment of the so-called ``extension lemma'' in Proposition \ref{propa1}. Extension 
lemmas have been intensively used for homogenization problems in perforated domains 
before, see e.g. \cite{JKO} and further references therein, as is briefly reviewed by us 
below too, see \eqref{what} and accompanying discussion. 

To proceed, notice first that the above approximation $U_\ve u^0$ as defined by 
\eqref{eq:uapr} lies in appropriate
functional spaces, in particular $U_\ve u^0\in H^1({\mathbb R}^n)$. 
Observe to this end that $u_0$ is infinitely smooth in $\Omega_2$ and
$\mathbb{R}^n\backslash\overline{\Omega_2}$ as a solution of
elliptic equations with constant coefficients \eqref{eq:eq4} and \eqref{eq:eq5}
respectively. Next $u_0$ decays  exponentially  at infinity, as a decaying solution
of equation with constant coefficients \eqref{eq:eq5} outside $\Omega_2$, since
$\beta(\lambda_0)<0$ and hence the fundamental solution of \eqref{eq:eq5} in the whole
$\mathbb{R}^n$ is exponentially decaying. Further since,
by \eqref{Vy},  $v(x,x/\varepsilon)=u_0(x)V(x/\varepsilon)$,
where $V(y)$ specified by  \eqref{eq:V}  is an $H^1$ periodic
function and its restrictions to $Q_0$ and $Q_1$ are infinitely
smooth, we conclude that
$v(x,x/\varepsilon)$ is an exponentially
decaying function belonging to $H^1(\mathbb{R}^n\backslash\Omega_2)$.

We further aim at establishing error bounds in the energy norms, i.e. with respect to the 
quadratic forms \eqref{eq:form} below associated with the equation \eqref{problem42}.  
For this, we slightly alter the approximation $U_\ve u^0$ by 
adding to it in $\Omega_1^\ve$ the first-order corrector from the asymptotic expansion \eqref{ansatz1}, 
hence introducing the following corrected approximation: 
 \begin{equation} U_1^\ve(x)=
    \left\{
      \begin{array}{lll}
        U_\ve u^0(x)\,=\,u_0(x), & x\in \Omega_2\cup\tilde\Omega_0^\varepsilon, 
\\
        U_\ve u^0(x)\,=\,u_0(x)+v(x,x/\ve)=u_0(x)(1+V(x/\ve)),
				& x\in
           \Omega_0^\varepsilon,
          \\
        U_\ve u^0(x)+\varepsilon N_j(x/\varepsilon)u_{0,j}(x)=
				u_0(x)+\varepsilon N_j(x/\varepsilon)u_{0,j}(x), 
		& x\in
            \Omega_1^\varepsilon.
      \end{array} \right.
      \label{eq:u3}
 \ee
Here $\varepsilon N_j(x/\varepsilon)u_{0,j}(x)$ is the first order
corrector, 
see
\eqref{u1}-\eqref{eq:N}. 

Consider now for any $\ve>0$ the quadratic form $b_\varepsilon$  
corresponding to the operator
$A_\varepsilon +I$:
 \be b_\varepsilon(w,u):=\sum_{j=0}^2
    \int_{\Omega_j^\varepsilon} a_j(\varepsilon)\nabla w\cdot\nabla u
    \,dx
   +
    \int_{\tilde{\Omega}_0^\varepsilon} \tilde a_0(\varepsilon)\nabla w\cdot\nabla u
    \,dx +\int_{\mathbb{R}^n} w \, u \,dx,
   \label{eq:form}
 \end{equation}
(for brevity of notation, $\Omega_2^\ve:=\Omega_2$, $a_j(\ve):=a_j, j=1,2$).
In particular, for the actual solution $\tilde u^\ve$ of (\ref{problem42}),
   \be
                   b_\varepsilon(\tilde{u}^\varepsilon,w)=(f^\varepsilon,w)_{\mathbb{R}^n},
                   \ \forall w \in H^1(\mathbb{R}^n),
        \label{eq:weeksol}
   \ee
where
\[
(f^\varepsilon,w)_{\mathbb{R}^n}:=\int_{\mathbb{R}^n} f^\ve w \,dx,
\]
and
  \be
       f^\varepsilon:=(\lambda_0+1)U_\varepsilon u^0=
\left\{\begin{array}{ll}(\lambda_0+1)u_0(x)\biggl(1+V(x/\ve)\biggr), &x\in
\Omega_0^\ve \\
(\lambda_0+1)u_0(x), &x\notin \Omega_0^\ve,
\end{array} \right.
        \label{eq:deff}
  \ee
via \eqref{problem42}, \eqref{eq:uapr} and \eqref{Vy}.

The domain of the form \eqref{eq:form} is $H^1(\mathbb{R}^n)$,
however we extend it to all ``piecewise $H^1$'' functions $w$, i.e.
such that $w\in H^1(\Omega_j^\varepsilon)$, $j=0,1,2$, $w\in
H^1(\tilde\Omega^\varepsilon_0)$, for which $b_\varepsilon(w,w)$, as
directly defined by the right hand side of \eqref{eq:form},  is
bounded. In particular, $U_1^\ve$ is in this ``extended'' domain.

The proof of Theorem \ref{l:estim} will be  based on the
following key technical lemma. We first state the lemma, then prove the
theorem assuming it is valid, and then prove the lemma itself (which
will in turn consist of several technical steps).
\begin{lemma}\label{lemmatech}
There exists an $\ve$ and $w$-independent $C>0$ such that for any sufficiently small $\ve>0$
and for any $w\in H^1(\mathbb{R}^n)$
\be
\left\vert b_\ve\left(w, U_1^\ve-\tilde u^\ve\right)\right\vert\,\leq\,
C\ve^{1/2}b_\ve(w,w)^{1/2}.
\label{lemmatech1}
\ee
\end{lemma}

{\bf Proof of Theorem \ref{l:estim}:}

Assume Lemma \ref{lemmatech} is valid. For any $\ve>0$, select 
$w=w^\ve(x):=U_2^\ve(x)-\tilde u^\ve(x)$, where $U_2^\ve$ is another
``corrected'' approximation constructed as follows:
 \begin{equation} U^\ve_2(x)=
    \left\{
      \begin{array}{ll}
        U_\varepsilon u^0(x)\,=u_0(x), & x\in \Omega_2,
          \\
           U_\varepsilon u^0(x)+\varepsilon\chi_\varepsilon(x)N_j(x/\varepsilon)u_{0,j}(x), 
& x\in \Omega_1^\varepsilon \cup
           \Omega_0^\varepsilon\cup\tilde\Omega_0^\varepsilon. 
      \end{array} \right.
      \label{eq:3}
\end{equation}
Here $\varepsilon N_j(x/\varepsilon)u_{0,j}(x)$ is the first order
corrector, constructed everywhere outside the defect, with (for example) a harmonic extension of 
$N(y)$ onto $Q_0$, 
\be
\chi_\varepsilon(x):=
\chi(\text{dist}(x,\partial\Omega_2)\varepsilon^{-1})
\label{chi}
\ee
 and $\chi(t)$ is a cut-off
function:  $\chi \in
C^\infty(\mathbb{R})$; $\chi(t)=0, t<1/2$ and $\chi(t)=1, t>1$.
Notice that both $U^\ve_2$ and $\tilde u^\ve$ are
in $H^1(\mathbb{R}^n)$. Hence by Lemma \ref{lemmatech},
\be
\left\vert b_\ve\left(U^\ve_2-\tilde u^\ve, U_1^\ve-\tilde u^\ve\right)\right\vert\,\leq\,
C\ve^{1/2}b_\ve(U^\ve_2-\tilde u^\ve,U^\ve_2-\tilde u^\ve)^{1/2},
\label{prth1}
\ee
with $C$ denoting henceforth constants  independent of
 $\varepsilon$ whose precise value is insignificant and can change form line to line.

On the other hand, by the non-negativity of the ``extended'' quadratic form, obviously,
 \be
             b_\varepsilon(U^\ve_2-\tilde{u}^\varepsilon,U^\ve_2-\tilde{u}^\varepsilon)\,\leq\,
             2\left|b_\varepsilon(U^\ve_2-\tilde{u}^\varepsilon,U^\ve_1-\tilde{u}^\varepsilon)\right|\,+\, 
b_\varepsilon(U^\ve_1-U^\ve_2, U^\ve_1-U^\ve_2).
   \label{eq:est3}
   \ee
Notice next that from \eqref{eq:u3} and \eqref{eq:3}
\be
U^\ve_1(x)-U^\ve_2(x)=
    \left\{
      \begin{array}{lll}
        0, & x\in \Omega_2, 
\\
        -\chi_\varepsilon(x)\varepsilon N_j(x/\varepsilon)u_{0,j}(x), & x\in
           \Omega_0^\varepsilon\cup\tilde\Omega_0^\varepsilon,
          \\
        (1-\chi_\varepsilon(x))\varepsilon N_j(x/\varepsilon)u_{0,j}(x), 
			& x\in
            \Omega_1^\varepsilon.
      \end{array} \right.
      \label{eq:u3new}
 \ee
Then, due to the small size (of order $\ve$ near $\partial\Omega_2$) of the support of
$1\,-\,\chi(\text{dist}(x,\partial\Omega_2)\varepsilon^{-1})$ as well as of $\tilde\Omega_0^\ve$,
and to the regularity and exponential decay of $u_0(x)$, 
we conclude that
   \be
             b_\varepsilon(U^\ve_1-U^\ve_2,U^\ve_1-U^\ve_2)
             \leq C \varepsilon.
   \label{eq:est2}
   \ee
Combining \eqref{eq:est3} with \eqref{prth1} and \eqref{eq:est2} implies:
\[
b_\varepsilon(U^\ve_2-\tilde{u}^\varepsilon,U^\ve_2-\tilde{u}^\varepsilon)\leq
2C\ve^{1/2}b_\varepsilon(U^\ve_2-\tilde{u}^\varepsilon,U^\ve_2-\tilde{u}^\varepsilon)^{1/2}\,+\,C\ve\,\leq
\]
\[
\frac{1}{ 2}b_\varepsilon(U^\ve_2-\tilde{u}^\varepsilon,U^\ve_2-\tilde{u}^\varepsilon)+
(2C^2+C)\ve,
\]
which yields, via \eqref{eq:form}, 
\be
\Vert U^\ve_2-\tilde{u}^\varepsilon\Vert^2_{L^2(\mathbb{R}^n)}\leq
b_\varepsilon(U^\ve_2-\tilde{u}^\varepsilon,U^\ve_2-\tilde{u}^\varepsilon)\leq C \ve.
\label{thpr2}
\ee
Notice finally that from \eqref{eq:3}, the boundedness of $N_j$ and $\chi_\ve$
as well as boundedness and exponential decay of $u_{0,j}$, 
\[
\Vert U_\ve u^0- U^\ve_2\Vert^2_{L^2(\mathbb{R}^n)}\,\leq \,C \ve.
\]
This together with \eqref{thpr2} implies via the triangle inequality that
\[
\Vert U_\ve u^0- \tilde{u}^\varepsilon\Vert_{L^2(\mathbb{R}^n)}\leq \,C \ve^{1/2},
\]
with appropriate constant $C$. This establishes \eqref{eq:main} and hence proves
the theorem. $\square$
\vspace{.2in}

{\bf Proof of Lemma \ref{lemmatech}:}

First, using \eqref{eq:weeksol} the entity in the left hand side of \eqref{lemmatech}
can be evaluated  as follows:
  \be
        b_\varepsilon(w,U_1^\ve-\tilde{u}^\varepsilon)\,=\, 
        b_\varepsilon(w,U_1^\ve)-(w,f^\varepsilon)_{\mathbb{R}^n}\,=\,I_1(\ve)\,+\,I_2(\ve),
 \label{eq:est6}
 \ee
where
 \be
        I_1(\ve)\,:=\,\,b_\varepsilon(w,U_1^\ve)
           \,=\,\sum_{j=0}^2
                       \int_{\Omega_j^\varepsilon} a_j(\varepsilon)\nabla w\cdot\nabla U_1^\ve
                       \,dx\,
                       +\int_{\tilde{\Omega}_0^\varepsilon} \tilde a_0(\varepsilon)\nabla w\cdot\nabla
U_1^\ve \,dx\,
                       +\int_{\mathbb{R}^n}wU_1^\ve \,dx,
 \label{eq:i1}
 \ee
and, via \eqref{eq:deff},
\be
I_2(\ve):=
     -(w,f^\varepsilon)_{\mathbb{R}^n}=
    -(\lambda_0+1)
    \left(\int_{\Omega_0^\varepsilon} w u_0(1+V)dx
    +\int_{\tilde{\Omega}_0^\varepsilon}w u_0dx
    +\int_{\Omega_2} w u_0dx
    +\int_{\Omega_1^\varepsilon}wu_0dx\right).
 \label{eq:500}
 \ee

It is further convenient to break $I_1(\ve)$ into four separate terms for the
four integration domains:
\be
I_1(\ve)=\,\tilde A_0(\ve) + \sum_{j=0}^2 A_j(\ve),
\label{i1four}
\ee
where
 \be
        \tilde A_0(\ve)\,:\,=
\int_{\tilde\Omega_0^\varepsilon} \tilde a_0(\ve)\nabla w\,\cdot\nabla U_1^\ve
                       dx
                       \,+\int_{\tilde\Omega_0^\varepsilon} wU_1^\ve
                       \,dx,
 \label{eq:tildea0}
 \ee
and
 \begin{eqnarray}
        A_0(\ve)&:\,=&\int_{\Omega_0^\varepsilon} \varepsilon^2a_0\nabla w\,\cdot\nabla U_1^\ve
                       dx
                       \,+\int_{\Omega_0^\varepsilon} wU_1^\ve
                       dx,
 \label{eq:a0}\\
 A_j(\ve)&:\,=&\int_{\Omega_j^\varepsilon} a_j\nabla w\,\cdot\nabla U_1^\ve
                       dx
                       \,+\int_{\Omega_j^\varepsilon} wU_1^\ve
                       dx, \,\,\,j=1,2.
\label{eq:a12}
 \end{eqnarray}
We will be separately estimating $\tilde A_0(\ve)$, $A_j(\ve)$,
$j=0,1,2,$ and then $I_2(\ve)$ in the series of the following
propositions, and will subsequently derive \eqref{lemmatech1} by
combining all these estimates.

\bigskip
\bigskip

\begin{prop}\label{proptildea0}
 \be
    \left|\tilde A_0(\ve)\right|
    \leq C\varepsilon^{1/2} b_\varepsilon(w,w)^{1/2}. 
 \label{eq:tilde}
 \ee
\end{prop}
\begin{proof}
Notice that the measure of  $\tilde\Omega_0^\varepsilon$ is bounded by
$C\varepsilon$. As a result, using 
Cauchy-Schwartz inequality,
\begin{eqnarray*}
    \left|\int_{\tilde{\Omega}_0^\varepsilon}
          \tilde a_0(\varepsilon)\nabla w\cdot\nabla U_1^\ve \,dx\right|&\leq&
\left(\int_{\tilde{\Omega}_0^\varepsilon}
          \tilde a_0(\varepsilon)\nabla w\cdot\nabla w\,dx\right)^{1/2}
\left(\int_{\tilde{\Omega}_0^\varepsilon}
         \tilde a_0(\varepsilon)\left|\nabla U_1^\ve\right|^2 dx\right)^{1/2}\leq
\nonumber\\
&&
C\varepsilon^{1/2} b_\varepsilon(w,w)^{1/2},
\end{eqnarray*}
where we have used \eqref{tildea0bounds} and the  $L^\infty$-boundedness of
$\nabla U_1^\ve$ in $\tilde\Omega_0^\ve$ via \eqref{eq:u3} and the boundedness of $u_0$. 
The second integral in
\eqref{eq:tildea0} is bounded similarly, which leads to
\eqref{eq:tilde}.
\end{proof}
\vspace{.15in}

Consider next the integrals over $\Omega_1^\ve$ in \eqref{eq:a12}:
 \be
        A_1(\ve)=\int_{\Omega_1^\varepsilon} a_1\nabla w\cdot\nabla U_1^\ve 
                       dx
                       +\int_{\Omega_1^\varepsilon} w U_1^\ve
                       dx.
 \label{eq:a1}
 \ee
Before formulating the corresponding result for $A_1(\ve)$, we need
to use the following technical construction. One can extend any function
$w$ from
$H^1\left(\mathbb{R}^n\setminus(\Omega_0^\varepsilon\cup\tilde{Q}_0^\varepsilon)\right)$
(let us remind that $\tilde Q_0^\ve$ is the set of all the
inclusions in $Q_0^\varepsilon$ which
 intersect with the boundary $\partial\Omega_2$ of $\Omega_2$, see Section 2) into the whole of $H^1\left(\mathbb{R}^n\right)$,
controlling its norm ``uniformly'' with respect to $\ve$. More
precisely, for any $\ve$ and any $w\in H^1\left(\mathbb{R}^n\setminus(\Omega_0^\varepsilon\cup\tilde{Q}_0^\varepsilon)\right)$
 there exists a function $\hat{w}\in H^1(\mathbb{R}^n)$ such that
\be
\hat{w}(x)=w(x), \
    x\in \mathbb{R}^n\setminus(\Omega_0^\varepsilon\cup\tilde{Q}_0^\varepsilon) \ \ \text{and}\ \
    \|\hat{w}\|_{H^1(\mathbb{R}^n)}\leq C
    \|w\|_{H^1(\mathbb{R}^n\setminus(\Omega_0^\varepsilon\cup\tilde{Q}_0^\varepsilon))},
\label{what} \ee where $C$ does not depend on $\ve$ and $w$.  The above
follows e.g. via a straightforward modification of the so called
``extension lemma'', see e.g. \cite{JKO} \S 3.1 Lemma 3.2 which uses the extension
construction, see for the latter e.g. \cite{St} \S 6.3.1, p.181, Theorem 5.

\begin{prop}\label{propa1}

 \be
        A_1(\ve)=(\beta(\lambda_0)+|Q_1|)\int_{\mathbb{R}^n\backslash\Omega_2}\hat{w }u_0dx
        +\int_{\partial\Omega_2}\hat{w} n_iA_{ij}^{\rm hom}u_{0,j}dS
        +\hat{{A}}_1(\ve),
 \label{eq:a10300}
 \ee
where
 \be
    |\hat{{A}}_1(\ve)|\leq C\varepsilon^{1/2} b_\varepsilon(w,w)^{1/2}.
\label{b32}
 \ee

\end{prop}
\begin{proof}
1. Let  $\mu^\varepsilon$ be characteristic function of
$\Omega_1^\varepsilon$. We can rewrite the
first integral in \eqref{eq:a1} via \eqref{eq:u3}  
as follows
 \be
    \int_{\Omega_1^\varepsilon} a_1\nabla w\cdot\nabla U_1^\ve
                       dx=
                       \int_{\mathbb{R}^n\backslash\Omega_2} \nabla
                       \hat w\cdot p_1^\varepsilon
                       dx,
    \label{eq:int1000}
 \ee
where
 \begin{equation}
        (p_1^\varepsilon(x))_i:=\mu^\varepsilon(x) a_1\biggl(u_{0,i}(x)+ N_{j,i}(y)u_{0,j}(x)
         +\varepsilon N_{j}(y)u_{0,ji}(x)\biggr), 
				\ x\in \mathbb{R}^n\backslash\Omega_2 , \,\,y=x/\varepsilon.
 \label{eq:p1}
 \ee
(Henceforth $p_i$ denotes the $i$-th component of the appropriate vector 
field $p(x)$.)
 The above flow $(p_1^\varepsilon(x))_i$ can be re-written in the following form
\begin{equation}
        (p_1^\varepsilon(x))_i=A_{ij}^{\rm hom}u_{0,j}+g_i^j(x/\varepsilon)u_{0,j}+\ve\mu^\varepsilon a_1
         N_{j}(x/\varepsilon)u_{0,ji}(x).
\label{eq:p1001}
 \ee
Here
 \be
        g_i^j(y):=\mu(y) a_1(\delta_{ij}+
        N_{j,i}(y))-A_{ij}^{\rm hom}, \ \ \  y\in Q,
\label{gij}
 \ee
$\mu$ is the characteristic function of $Q_1$, and $A_{ij}^{\rm hom}$ are the entries
of the homogenized matrix $A^{\rm hom}$, see \eqref{eq:A}, and
 $N_j(y)$ are here assumed extended by zero on $Q_0$. It follows then from
\eqref{eq:N}  that vector field $g_i^j\in \left[L^2(Q)\right]^n$ (with fixed $j$) is divergence
free in the whole of the periodicity cell $Q$ in the following (weak) sense:
 \be
 \int_Q \,\partial_i\psi\,  g_i^j\,dy\,=\,0,\,\, \ \  \forall \psi \in
 H^1(\square) 
\label{divfree}
 \ee
($ H^p(\square)$ stands for the closure in  $H^p(Q)$ of all $Q$-periodic $C^\infty$ 
functions and $\partial_i:=\partial/\partial y_i$). Since \eqref{gij} and \eqref{eq:A} also 
imply that $g_i^j$ have zero mean value over $Q$, they  can be 
rewritten as ``divergences'' of skew-symmetric field $G^j_{ik}(y)\in H^1(\square)$ (which is in fact a 
generalization of the curl operation to arbitrary $n$, see e.g. \cite{JKO} \S 1.1 pp. 6-7):
  \be
        g_i^j(y)=\partial_k G_{ik}^j(y),\ \ \ \ 
        G_{ik}^j(y)\,=\,-\,\,G_{ki}^j(y). 
 \label{eq:G}
 \ee\

 Consequently \eqref{eq:p1001} can be rewritten as
\begin{equation}
        (p_1^\varepsilon(x))_i=A_{ij}^{\rm hom}u_{0,j}(x)+\varepsilon
\frac{\partial}{ \partial x_k} \biggl(G_{ik}^j(x/\ve)u_{0,j}(x)\biggr)+\ve\mu^\varepsilon a_1
       N_{j}(x/\varepsilon)u_{0,ji}(x)\,\,-
\label{eq:p1020}
 \ee
 $$
\varepsilon G_{ik}^j(x/\ve)u_{0,jk}(x), \,
         \, x\in \mathbb{R}^n\backslash\Omega_2. 
 $$

Function $\frac{\partial}{\partial x_k} (G_{ik}^ju_{0,j})$ from \eqref{eq:p1020} is
not divergence free in $\mathbb{R}^n\backslash \Omega_2$.
We introduce its  ``divergence free modification''
$\frac{\partial}{\partial x_k} (\chi_\ve G_{ik}^ju_{0,j})$
 which differs from $\frac{\partial}{\partial x_k} (G_{ik}^ju_{0,j})$
insignificantly by employing again the cut-off function
$\chi_\ve(x)$, see  \eqref{chi}. As a result,
\begin{equation}
        (p_1^\varepsilon(x))_i=A_{ij}^{\rm hom}u_{0,j}(x)+
\varepsilon\frac{\partial}{\partial x_k} \biggl(\chi_\varepsilon G_{ik}^j(x/\ve)u_{0,j}(x)\biggr)+
\varepsilon\mu^\varepsilon a_1
          N_{j}(x/\varepsilon)u_{0,ji}(x)\,+ \biggr.
\label{eq:p1030}
 \ee
 $$
 \biggl.   
		\left({{r_1}}^\varepsilon(x)\right)_i, \,\,\,
         \, x\in \mathbb{R}^n\backslash\Omega_2,
 $$
where
 \be
    \left({{r_1}}^\varepsilon(x)\right)_i\,=\, \varepsilon\frac{\partial}{\partial x_k}
\biggl((1-\chi_\varepsilon(x) )G_{ik}^j(x/\ve)u_{0,j}(x)\biggr)\,-\,
    \varepsilon \,G_{ik}^j(x/\ve)u_{0,jk}(x). 
 \label{eq:r1010}
 \ee 

2. Integrating by parts in \eqref{eq:int1000} and using \eqref{eq:p1030}, \eqref{eq:G} 
and \eqref{divfree}
we obtain
 \be
    \int_{\Omega_1^\varepsilon } a_1\nabla w\cdot\nabla U_1^\ve
                       dx=
                       \int_{\mathbb{R}^n\backslash\Omega_2} \nabla
                       \hat{w}\cdot p_1^\varepsilon
                       dx=
                       -\int_{\mathbb{R}^n\backslash\Omega_2}
                       \hat{w} A_{ij}^{\rm hom}u_{0,ij}dx
  \label{eq:int10100}
 \ee
 $$
    +\int_{\partial\Omega_2}
                       \hat{w} n_iA_{ij}^{\rm hom}u_{0,j}dS
\ \ + \ \ R^\ve_1\,+\,R^\ve_2, \,\,\,\,\, \ \  
 $$
with the ``remainders''
 \be
    R^\ve_1\,:=\,\int_{\mathbb{R}^n\backslash\Omega_2}
           \nabla \hat w\cdot {{r_1}}^\varepsilon
        dx,
    \label{eq:R1}
 \ee
 \be
    R^\ve_2\,:=\,\int_{\Omega_1^\varepsilon}
                     w_{,i}\, \varepsilon a_1\,
        N_{j}(x/\ve)u_{0,ji}(x)\,dx. 
 \label{eq:R2}
 \ee

We argue that both remainders are ``small''. 
 Let us consider $R^\ve_1$.  The only term of ``order one'' in $R_1^\ve $ is that corresponding to 
$\varepsilon\frac{\partial}{\partial x_k}\left((1-\chi^\varepsilon
)G_{ik}^j
\right)$, see  
\eqref{eq:r1010}, however the size of the support of 
$(1-\chi^\varepsilon)$ is of order $\varepsilon$. 
So we can apply the Cauchy-Schwartz inequality to each term for \eqref{eq:r1010} in \eqref{eq:R1} 
and use the fact 
that $G^j_{ik}$ is $\ve$-periodic and hence its $H^1$ norm over the support of $(1-\chi^\varepsilon)$ 
is bounded by $C\ve^{-1/2}$. As a result, 
\be
 |R^\ve_1|\leq \varepsilon^{1/2}\|\hat w\|_{H^1({\mathbb R}^n)}.
\label{r1eps0}
\ee
(Having also noticed that the remaining, ``order $\ve$'', terms in $R^\ve_1$ contribute only
order $\ve$ terms into the right hand side of \eqref{r1eps0}, upon straightforward
application of the Cauchy-Schwartz inequality and the exponential decay of $u_0$.) 
Finally, combining \eqref{r1eps0} with the extension bound \eqref{what} we conclude that 
\be
 |R^\ve_1|\,\leq \,C\varepsilon^{1/2}b_\varepsilon( w, w)^{1/2}.
\label{r1eps}
\ee

Notice next that 
$R^\ve_2$ is of order $\ve$ with the exponentially
decaying $u_0$, which implies $|R^\ve_2|\leq \varepsilon b_\varepsilon(w,w)^{1/2}$.

3. Summarising, we obtain following estimate on the combined smallness of both remainders:
 \be
    |R^\ve_1+R^\ve_2|\leq C\varepsilon^{1/2} b_\varepsilon(w,w)^{1/2}.
		\label{539}
 \ee

\vspace{.2in}

As a result we can evaluate $A_1(\ve)$ (see \eqref{eq:a1} and \eqref{eq:u3}) 
 as follows
 \be
    A_1(\ve)=
        -\int_{\mathbb{R}^n\backslash\Omega_2} \hat{w} A_{ij}^{\rm hom}u_{0,ij}dx\,
        +\,\int_{\Omega_1^\varepsilon}w u_0\,dx
        +
 \label{eq:a10010}
 \ee
 $$
        +\int_{\partial\Omega_2}\hat w n_iA_{ij}^{\rm hom}u_{0,j}dS
        +\tilde{A}_1(\ve),
 $$
where
 $$
    \tilde{A}_1(\ve):=
    R^\ve_1+R^\ve_2
		+\int_{\Omega_1^\varepsilon} w 
		\varepsilon N_ju_{0,j}\,dx, 
 $$
and consequently, via \eqref{539} and the straightforward estimate for the last integral, 
 \be
    |\tilde{A}_1(\ve)|\leq C\varepsilon^{1/2} b_\varepsilon(w,w)^{1/2}.
 \ee
\bigskip

4. Let us consider the integral over $\Omega_1^\varepsilon$
in \eqref{eq:a10010}:
 \be
    \int_{\Omega_1^\varepsilon}w u_0\,dx\,\,=\,\int_{\mathbb{R}^n\backslash\Omega_2}\mu^\varepsilon w
    u_0\,dx,
 \label{eq:1000}
 \ee
regarding here $\mu^\varepsilon(x)=\mu(x/\varepsilon)$ as the characteristic
function of $\mathbb{R}^n\backslash Q_0^\varepsilon$, see Section 2. In a standard way, the associated $Q$-periodic characteristic function $\mu(y)$ can be
presented as follows:
 \be
    \mu(y)=|Q_1|+\triangle_y M(y), \ \ \  y\in Q,
		\label{543}
 \ee
where $M\in H^2(\square)$.
 Then \eqref{eq:1000} can be evaluated in the following way:
  \be
    \int_{\mathbb{R}^n\backslash\Omega_2}
          \mu^\varepsilon w u_0\,dx\,
    =\int_{\mathbb{R}^n\backslash\Omega_2}
         \mu^\varepsilon \hat{w} u_0\,dx=\,
    \int_{\mathbb{R}^n\backslash\Omega_2}
          (|Q_1|+\varepsilon^2\triangle_x M(x/\varepsilon))\hat{w}u_0\,dx\,=
 \label{eq:1010}
 \ee
\be
      |Q_1|\int_{\mathbb{R}^n\backslash\Omega_2}
        \hat{w}u_0\,dx\,-
    \int_{\mathbb{R}^n\backslash\Omega_2}
          \varepsilon^2\nabla_x
          M(x/\varepsilon)\cdot\nabla_x(\hat{w}u_0)dx\,+
    \int_{\partial\Omega_2}
         \varepsilon^2 n_i\left(\frac{\partial}{\partial x_i}M(x/\varepsilon)\right)\hat{w}u_0\,dS.
\label{b58}
\ee
Following the pattern of the previous estimates (i.e. again using the boundedness and 
the exponential decay of $u_0$) the second integral on the right hand side of \eqref{b58} 
can be readily bounded by $C\varepsilon b_\varepsilon(w,w)^{1/2}$. For the last integral, 
$$\left|\int_{\partial\Omega_2}
         \varepsilon^2 n_i\left(\frac{\partial}{\partial x_i}M(x/\varepsilon)\right)\hat{w}u_0\,dS\right|\leq 
				C\int_{\partial\Omega_2} \ve^2\left|\nabla_xM(x/\ve)\right|\,\left|\hat w\right|\,dS\,= 
 $$
 $$
C\ve\int_{\partial\Omega_2} \left|\nabla_yM(x/\ve)\right|\,\left|\hat w\right|\,dS\,\leq\, 
C\ve\|\nabla_yM\|_{L^\infty(Q)}\|\hat w\|_{L^2(\partial\Omega_2)}\,\leq\,
   C \varepsilon b_\varepsilon (w,w)^{1/2}. 
   $$
Here we have used the classical property of the $L^\infty$-boundedness of $\nabla_yM$ for $M(y)$ solving 
the Laplace equation \eqref{543}, the boundedness of $|\partial\Omega_2|$, 
continuity of the trace operator from
$H^1(\Omega_2)$ into $L_2(\partial\Omega_2)$, and \eqref{what}.

Hence the last two terms in \eqref{b58} can be estimated as follows:
 \be
    \left|
    \int_{\partial\Omega_2}
         \varepsilon^2 n_i\left(\frac{\partial}{\partial x_i}M(x/\varepsilon)\right)\hat{w}u_0dS-
    \int_{\mathbb{R}^n\backslash\Omega_2}
          \varepsilon^2\nabla_x
          M(x/\varepsilon)\cdot\nabla_x(\hat{w}u_0)dx\right|
    \leq C\varepsilon b_\varepsilon(w,w)^{1/2}.
\label{eq:10101}
 \ee

5. As a result we have
 \be
    A_1(\ve)=
        -\int_{\mathbb{R}^n\backslash\Omega_2} \hat{w }A_{ij}^{\rm hom}u_{0,ij}\,dx
       \,+\,|Q_1|\int_{\mathbb{R}^n\backslash\Omega_2}\hat{w }u_0dx
        +
 \ee
 $$
        +\int_{\partial\Omega_2}\hat{w} n_iA_{ij}^{\rm hom}u_{0,j}dS
        +\hat{{A}}_1(\ve),
 $$
with $\hat{{A}}_1(\ve)$ satisfying \eqref{b32}.
Finally, using equation \eqref {eq:eq5} for $u_0$ 
we obtain \eqref{eq:a10300}.
\end{proof}

\vspace{.1in}

\begin{prop}\label{propa0}
  \be
        A_0(\ve)=\int_{\Omega_0^\varepsilon} w   f^\varepsilon
             dx-
				\left(\beta(\lambda_0)-|Q_1|\lambda_0\right)
				\int_{\mathbb{R}^n\backslash\Omega_2}\hat{w }u_0dx
				+\hat{A}_0(\ve),\ \
         |\hat{A_0}(\ve)|\leq C \varepsilon b_\varepsilon(w,w)^{1/2}.
 \label{eq:estA0}
 \ee

\end{prop}
\begin{proof}
Consider the flows associated with $U_1^\ve$  in \eqref{eq:a0},
i.e. let $p_0^\varepsilon(x):=a_0 \varepsilon^2\nabla U_{1}^\ve(x)$, $
x\in \Omega_0^\varepsilon$. Using \eqref{eq:u3}, cf. \eqref{eq:uapr}
and \eqref{Vy},
 \begin{equation}
      (p_0^\varepsilon(x))_i=a_0\varepsilon
         V_{,i}(x/\varepsilon)u_0(x)+a_0\ve^2(r_0^\varepsilon(x,x/\varepsilon))_i,
				\,\, \ \ 
				x\in
        \Omega_0^\varepsilon,  
 \label{eq:p0}
 \ee
where
 \be
         (r_0^\varepsilon(x,y))_i:=u_{0,i}(x)(1+V(y)). 
 \ee

Then 
$A_0(\ve)$, see \eqref{eq:a0}, can be evaluated as follows:
 \be
       A_0(\ve)
        =a_0\int_{\Omega_0^\varepsilon} 
				w_{,i}\varepsilon V_{,i} u_0
                       dx
        +a_0\int_{\Omega_0^\varepsilon} \varepsilon^2 \nabla w\cdot r_0^\varepsilon
                       \,dx \,+\,\int_{\Omega_0^\varepsilon} w 
u_0(1+V)
                       dx.
 \label{eq:om}
 \ee
For the first integral in \eqref{eq:om} decompose $w$ into its extension $\hat w$, see \eqref{what}, and 
$z^\varepsilon:=w-\hat w$ and notice that $z^\varepsilon\in H_0^1(\Omega_0^\varepsilon)$. 
As a result, 
\be
a_0\int_{\Omega_0^\varepsilon}	w_{,i}\varepsilon V_{,i} u_0 \,dx = 
a_0\int_{\Omega_0^\varepsilon}	\hat w_{,i}\varepsilon V_{,i} u_0 \,dx + 
a_0\int_{\Omega_0^\varepsilon}	z^\varepsilon_{,i}\varepsilon V_{,i} u_0 \,dx. 
\label{552}
\ee
Applying Cauchy-Schwartz inequality to the first integral in \eqref{552} and using \eqref{what} we conclude that it is small: 

 \be
         \left|a_0\int_{\Omega_0^\varepsilon}  \hat w_{,i}\varepsilon V_{,i}
u_0
                       dx\right|
         \leq C\,\varepsilon\,\|\hat w\|_{H^1(\mathbb{R}^n)}
\leq C \ve b_\ve(w,w)^{1/2} 
 \ee
(having also used the boundedness of $u_0$ and $V$ and the exponential decay of $u_0$).

For the latter integral in \eqref{552}, 
upon integration by parts and using \eqref{eq:V}, 
$$
a_0\int_{\Omega_0^\varepsilon}	z^\varepsilon_{,i}\varepsilon V_{,i} u_0 dx= 
\lambda_0\int_{\Omega_0^\varepsilon}z^\varepsilon u_0(1+V)- 
a_0\int_{\Omega_0^\varepsilon}	z^\varepsilon \varepsilon V_{,i} u_{0,i} dx. 
$$
The last integral on the right hand side is bounded by $C\ve b_\ve(w,w)^{1/2}$ by 
splitting $z^\varepsilon=w-\hat w$ and using again the Cauchy-Schwartz inequality 
and \eqref{552}. 
Splitting in turn the other integral, 
$$
\lambda_0\int_{\Omega_0^\varepsilon}z^\varepsilon u_0(1+V)\,dx\,= \, 
\lambda_0\int_{\Omega_0^\varepsilon}w u_0(1+V)\,dx\,- \, 
\lambda_0\int_{\Omega_0^\varepsilon}\hat w u_0(1+V)\,dx\,=
$$
$$
\lambda_0
\int_{\Omega_0^\varepsilon}w 
u_0(1+V)\,dx \,
-\,\lambda_0 \left(|Q_0|+\langle V\rangle\right) 
\int_{\mathbb{R}^n\backslash\Omega_2}\hat{w }u_0\,dx\,+\,R_0^\varepsilon(x). 
$$
Here we have used the argument as in \eqref{eq:1010}: denoting by $\mu_0(y)$ the characteristic function of 
the inclusion $Q_0$, similarly to \eqref{543}, 
\[
\mu_0(y)(1+V(y))\,=\, (|Q_0|+\langle V\rangle)\,+\,\Delta_y W(y), \ \ \ y\in Q,
\]
for some $W\in H^2(\square)$, and then proceeding with evaluation of 
$\int_{\Omega_0^\varepsilon}\hat w u_0(1+V)\,dx$ as in \eqref{eq:1010}. 
As a result, 
$  \left|  R_0^\varepsilon(x) \right|\leq C \varepsilon b_\varepsilon(w,w)^{1/2}$. 

Notice finally that the remainder term in \eqref{eq:om} is also small: 
$$
        \left|\int_{\Omega_0^\varepsilon} \varepsilon^2\nabla w\cdot r_0^\varepsilon
                       dx\right|
        \leq C\,\varepsilon\, b_\varepsilon(w,w)^{1/2},
$$
(having again used the boundedness of $u_0$ and $V$ and the exponential decay of $u_0$). 

As a result of the above estimates,  \eqref{eq:om} yields 
\be
A_0(\ve)\,=\,
 \lambda_0
\int_{\Omega_0^\varepsilon}w 
u_0(1+V)\,dx \,
-\,\lambda_0 \left(|Q_0|+\langle V\rangle\right) 
\int_{\mathbb{R}^n\backslash\Omega_2}\hat{w }u_0\,dx\,+
\,\int_{\Omega_0^\varepsilon} w 
u_0(1+V)
                       dx\,+\,
\hat A_0(\ve), 
\label{555}
\ee
with $\hat A_0(\ve)\leq \varepsilon^{1/2}\, b_\varepsilon(w,w)^{1/2}$. 

Finally, using  \eqref{eq:deff} and \eqref{betadef}, we 
observe that  \eqref{555} yields \eqref{eq:estA0}.
\end{proof}

\bigskip
\bigskip

Consider now the integrals $A_2(\ve)$ over the defect domain $\Omega_2$, see \eqref{eq:a12}:
\begin{prop}\label{propa2}
 \be
    A_2(\ve)=
        \int_{\Omega_2}w f^\varepsilon dx
        -\int_{\partial\Omega_2}\hat{w}
        n_ia_2u_{0,i}dS+\hat{A}_2(\ve),
 \label{eq:a20300}
 \ee
where $f^\ve$ is given by \eqref{eq:deff}, and 
 \be
    |\hat{A}_2(\ve)|\leq C\varepsilon^{1/2} b_\varepsilon(w,w)^{1/2}.
\label{b63}
 \ee
\end{prop}
\begin{proof}
 \be
        A_2(\ve)=\int_{\Omega_2} a_2\nabla w\cdot\nabla U_1^\ve
                       dx
                       +\int_{\Omega_2} w U_1^\ve
                       dx=
 \label{eq:a2}
 \ee
 $$
        \int_{\Omega_2} a_2\nabla \hat{w}\cdot\nabla U_1^\ve
                       dx
                       +\int_{\Omega_2} \hat{w} U_1^\ve
                       dx+\tilde{{A}}_2(\ve),
 $$
where
 $$
     \tilde{{A}}_2(\ve)=\int_{\Omega_2\cap \tilde Q_0^\ve} a_2\nabla (w-\hat{w})\cdot\nabla U_1^\ve
                       dx
                       +\int_{\Omega_2\cap \tilde Q_0^\ve} (w-\hat{w}) U_1^\ve
                       dx.
 $$
 Integrating by parts and recalling \eqref{eq:u3}, we obtain (with the consistent choice
of the normal $n$ to $\partial\Omega_2$ being inward for $\Omega_2$)
 \be
    A_2(\ve)=
        -\int_{\Omega_2} \hat{w }a_2\triangle u_{0}dx
        +\int_{\Omega_2}\hat{w} u_0dx
        -\int_{\partial\Omega_2}\hat{w} n_ia_2u_{0,i}dS+\tilde{{A}}_2(\ve)=
 \ee
$$
       (\lambda_0+1)\int_{\Omega_2}\hat{w} u_0dx
        -\int_{\partial\Omega_2}\hat{w} n_ia_2u_{0,i}dS+\tilde{{A}}_2(\ve)=\int_{\Omega_2}\hat{w} f^\varepsilon dx
        -\int_{\partial\Omega_2}\hat{w} n_ia_2u_{0,i}dS+\tilde{{A}}_2(\ve),
$$
having used \eqref{eq:eq4} and \eqref{eq:deff}. The above expression can be rewritten as
\[
    A_2(\ve)=
        \int_{\Omega_2}w f^\varepsilon dx
        -\int_{\partial\Omega_2}\hat{w}
        n_ia_2u_{0,i}dS+\hat{A}_2(\ve),
\]
where
 $$\hat{A}_2(\ve):=\int_{\Omega_2\cap \tilde Q_0^\ve} a_2\nabla (w-\hat{w})\cdot\nabla U_1^\ve
                       dx
                       +\int_{\Omega_2\cap \tilde Q_0^\ve} (w-\hat{w}) U_1^\ve
                       dx+\int_{\Omega_2\cap \tilde Q_0^\ve}(\hat{w}-w) f^\varepsilon
                       dx.
 $$
Arguing further as in Proposition \ref{proptildea0}, we obtain \eqref{b63}.\end{proof}

\vspace{.1in}

We can now complete the proof of the Lemma \ref{lemmatech} with the aid of the
established Propositions \ref{proptildea0}--\ref{propa2} as follows.

Combine \eqref{eq:a20300} with  \eqref{eq:tilde}, \eqref{eq:a10300}, \eqref{eq:estA0}  which
are all substituted into \eqref{i1four}, and then employ \eqref{limsystem}. As a result, 
$I_1(\ve)$, see\eqref{eq:i1}, is evaluated as follows:
 \be
    I_1(\ve)=b_\varepsilon(w,U_1^\ve)=\int_{\Omega_2\bigcup\Omega_0^\varepsilon}w f^\varepsilon
    dx+(\lambda_0+1)|Q_1|\int_{\mathbb{R}^n\backslash \Omega_2}\hat{w} u_0
    dx+\hat{I}_1(\ve),
 \label{eq: esti1}
 \ee
where
 \be
    |\hat{I}_1(\ve)|\leq C\varepsilon^{1/2} b_\varepsilon(w,w)^{1/2}.
 \ee

Let us now consider $I_2(\ve)$, see \eqref{eq:500}. Noticing that the last term in the
right hand side of \eqref{eq:500} is a constant times \eqref{eq:1000}, we can
employ again \eqref{eq:1010}-\eqref{eq:10101} which results in:
 \be
    \int_{\Omega_1^\varepsilon }w
     u_0dx=|Q_1|
     \int_{\mathbb{R}^n\backslash\Omega_2}
          \hat{w}u_0dx+R^\ve,
 \label{eq:400}
 \ee
where
 \be
    |R^\ve|\leq C\varepsilon b_\varepsilon(w,w)^{1/2}.
 \ee
Notice next that the integral over $\tilde{\Omega}_0^\varepsilon$ in \eqref{eq:500} is
small due to the smallness of the measure of
$\tilde{\Omega}_0^\varepsilon$:
 \be
    \left|\int_{\tilde{\Omega}_0^\varepsilon}w u_0dx\right|
    \leq C\varepsilon^{1/2} b_\varepsilon(w,w)^{1/2}.
\label{hzchb}
 \ee

As a result of employing \eqref{eq:400}--\eqref{hzchb} in \eqref{eq:500} it
 can be rewritten in the following form:
 \be
    I_2(\ve)=
     -(\lambda_0+1)
    \biggl(\int_{\Omega_0^\varepsilon} w u_0(1+V)dx
    +\int_{\Omega_2} w u_0dx
    +|Q_1|
     \int_{\mathbb{R}^n\backslash\Omega_2}
          \hat{w}u_0dx\biggr)+\hat{I}_2(\ve),
 \label{eq:510}
 \ee
where
 \be
    |\hat{I}_2(\ve)|\leq C\varepsilon^{1/2} b_\varepsilon(w,w)^{1/2}.
 \ee
 Adding finally \eqref{eq:510} and
\eqref{eq: esti1} and using \eqref{eq:deff}, we conclude that \eqref{eq:est6} can be bounded as
follows:
 \be
        |b_\varepsilon(w,U_1^\ve-\tilde{u}^\varepsilon)|
        \leq C\varepsilon^{1/2} b_\varepsilon(w,w)^{1/2},
 \ee
which is identical to \eqref{lemmatech1} and hence proves Lemma
\ref{lemmatech}. $\square$

\section{An example}

Straightforward analysis of the limit problem
(\ref{eq:eq4})--(\ref{316}) shows that one can sometimes explicitly calculate isolated
eigenvalues of operator $A_0$ and associated eigenfunctions, at least in the case when 
$\Omega_2=\{x: |x|<R\}$ (i.e. the defect is a ball of some radius $R>0$) and
$A^{\rm hom}=a^{\rm hom}I$,  i.e. the perforated homogenized matrix is isotropic. The latter 
isotropy is the case e.g. when  the microscopic periodic inclusions $Q_0$ have appropriate symmetries, in particular being themselves 
balls in which case $\beta(\lambda)$ is in fact explicitly found in terms of Bessel functions and is in 
particular a simple trigonometric function for $n=3$, see \cite{BKS08}. 
We sketch the details below.

Under the above assumptions a solution to the spectral problem
(\ref{eq:eq4})--(\ref{316}) is sought by separation of variables in the
spherical coordinates $x=(r,\omega)$, $r:=|x|$, $\omega:=x/|x|\in S^{n-1}$, 
the unit sphere in ${\mathbb R}^n$ ($n\geq 2$):
\begin{equation}
u_0(x)\,=\,
\left\{
      \begin{array}{ll}
         r^{-\,(n-2)/2}J_m\left((\lambda/a_2)^{1/2}\,r\right)P_m(\omega), & |x|\leq R,
          \\
         \alpha\, r^{-\,(n-2)/2}I_m\left(|\beta(\lambda)/a^{\rm hom}|^{1/2}\,r\right)P_m(\omega), & |x|\geq R,
      \end{array} \right.,
\label{spher}
\end{equation}
assuming $\lambda>0$,  $\beta(\lambda)<0$. Parameters $m$, $\lambda$ and $\alpha$ are to be found. 
Here $J_m(z)$ and $I_m(z)$ are the Bessel and the modified Bessel functions respectively,
see e.g. \cite{AbrSt},
$P_m(\omega)$ are the spherical functions which are the eigenfunctions
of the Laplace-Beltrami operator $\Delta_\omega$ on $S^{n-1}$:
\begin{equation}
\left(\Delta_\omega\,+\,m^2-(n-2)^2/4\right)\,P_m(\omega)\,=\,0.
\label{LB}
\end{equation}

The spherical spectral problem (\ref{LB}) is a classical one and
defines explicit eigenvalues $m$ and eigenfunctions $P_m$ (e.g. for
$n=3$ the spectral parameter $m$ is half-integer and $P_m$, in spherical 
coordinates $(\theta, \varphi)$, are
products of the Legendre polynomials of $\cos\theta$ and trigonometric functions of $\varphi$; 
for $n=2$ those are trigonometric functions and $m$ is integer).
Having selected $m$ and associated $P_m(\omega)$ the function
$u_0(x)$ determined by (\ref{spher}) automatically satisfies the
equations (\ref{eq:eq4}) and (\ref{eq:eq5}), and exponentially
decays at infinity ($r\to\infty$) with the rate
$\exp\left(-|\beta(\lambda)/a^{\rm hom}|^{1/2}r\right)$, due to the exponential decay of $I_m(z)$. 
The remaining
parameters $\alpha$ and $\lambda$ (the eigenvalue) are determined
from (\ref{316}) at $r=R$, which specializes to:
\begin{equation}
J_m\left((\lambda/a_2)^{1/2}\,R\right)\,=\,\alpha\,I_m\left(|\beta(\lambda)/a^{\rm hom}|^{1/2}\,R\right)
\label{intcond1}
\end{equation}
\begin{equation}
(\lambda a_2)^{1/2}
J'_m\left((\lambda/a_2)^{1/2}R\right)-a_2\frac{n-2}{ 2 R}J_m\left((\lambda/a_2)^{1/2}R\right)
\,=\,
\label{intcond2}
\end{equation}
\[
=\,
\alpha \left|\beta(\lambda)a^{\rm hom}\right|^{1/2}
I'_m\left(|\beta(\lambda)/a^{\rm hom}|^{1/2}\,R\right)-\alpha a^{\rm hom}\frac{n-2}{ 2 R}
I_m\left(|\beta(\lambda)/a^{\rm hom}|^{1/2}\,R\right),
\]
where $J'_m$ and $I'_m$ denote derivatives of the relevant (modified) Bessel functions.

All $\lambda$ with $\beta(\lambda)<0$ for which there exists a solution to (\ref{LB}),
(\ref{intcond1})--(\ref{intcond2}) describe the point
spectrum of the operator $A_0$ in the ``gaps''. One can see that it
is generally non-empty. Say, for $n=3$ and $m=1/2$, $P_{1/2}(\omega)\equiv 1$ is an
eigenfunction of (\ref{LB}) (which determines the spherically symmetric solutions via 
(\ref{spher}) ), and $J_{1/2}$ and $I_{1/2}$ are represented by explicit trigonometric
and exponential functions respectively, e.g. \cite{AbrSt}. 
This allows replacing (up to insignificant multiplicative constants) the radial parts in the
right hand sides of \eqref{spher} by $r^{-1}\sin((\lambda/a_2)^{1/2}r)$ and
$r^{-1}\exp(-|\beta(\lambda)/a^{\rm hom}|^{1/2}r)$, transforming
 (\ref{intcond1})--(\ref{intcond2}) into:
\begin{equation}
\sin\left((\lambda/a_2)^{1/2}\,R\right)\,=\,\alpha\,\exp\left(-|\beta(\lambda)/a^{\rm hom}|^{1/2}\,R\right)
\label{intcond12}
\end{equation}
\[
(\lambda a_2)^{1/2}
\cos\left((\lambda/ a_2)^{1/2}R\right)
-\frac{a_2}{ R}\sin\left((\lambda/ a_2)^{1/2}R\right)
\,=
\]
\begin{equation}
-\alpha\, a^{\rm hom}
\left[|\beta(\lambda)/a^{\rm hom}|^{1/2}+\frac{1}{ R}\right]
\exp\left(-|\beta(\lambda)/a^{\rm hom}|^{1/2}R\right).
\label{intcond22}
\end{equation}
The condition of solvability of (\ref{intcond12})--(\ref{intcond22}) obviously reads:
\begin{equation}
\mbox{cotan}\left((\lambda/a_2)^{1/2}\,R\right)
+\frac{a^{\rm hom}-a_2}{(\lambda a_2)^{1/2}R}
\,=\,
-\,
\left(\frac{a^{\rm hom}|\beta(\lambda)|}{\lambda a_2}\right)^{1/2},\,\,\,\beta(\lambda)<0.
\label{radsymm}
\end{equation}
Noticing that the left hand side of \eqref{radsymm} is a function of $\lambda^{1/2}R$
one can easily see that
by varying $R>0$ one can obtain infinitely many solutions of (\ref{radsymm})
at any $\lambda$ in any gap ($\beta(\lambda)<0$)
 of the unperturbed linear operator.

Remark finally that, when $Q_0=\left\{y\in Q :\, |y-y_0|\leq \rho\,\right\}$ i.e. is a ball of some radius $\rho<1/2$, 
$\beta(\lambda)$ is itself explicitly found in terms of a simple radially symmetric solution of 
\eqref{eq:V} and then \eqref{betadef}, see \cite{BKS08}  p.419 Remark 1. In particular, 
for $n=3$,  
\be
\beta(\lambda)\,=\,\left(1\,-\,\frac{4}{3}\pi \rho^3\right)\lambda\,+\,
4\pi\rho\left(1\,-\,\rho\lambda^{1/2}\mbox{cotan}\left(\lambda^{1/2}\rho\right)\right).
\label{betaexplicit}
\ee

\section{Discussion: further refinement of the results.}

In this section we describe a 
further refinement of
the results formulated in Theorem \ref{thm1} using the results of Cherdantsev \cite{cherd}, who 
studied a similar problem applying an alternative 
technique of the 
two-scale convergence \cite{Ng, Al, Zh1, Zh2} and who followed in turn a general strategy of 
Zhikov \cite{Zh1,Zh2}. 
 Remark that, apart from the following important strengthening 
of the results 
(subject to an additional restriction on the boundary inclusions as below), the approach of \cite{cherd} 
potentially allows less regular boundaries for both the periodic inclusions $Q_0$ and for the 
defect $\Omega_2$. It is also, in principle, extendable for (periodically) variable $a_j(y)$, $j=0,1$,  
as well as variable $a_2(x)$ in the defect with some minimal assumptions on the regularity. Its 
disadvantage however is in the intrinsic inability of the method of two-scale convergence to provide 
error bounds, i.e. the rate of convergence as we do in Theorem \ref{thm1}, see \eqref{mainresult} and 
\eqref{eq28}. It may therefore be advantageous to combine the results based on the two methods, 
as discussed below.

For the results of \cite{cherd} to be applicable to our setting, the only but essential additional assumption which 
has to be made is restricting further our assumption \eqref{tildea0bounds} on the boundary inclusions. 
Namely, again not pursuing here the maximal generality, following formula (2.5) of \cite{cherd}  we now require 
\be
    \hat A_0\ve^{2-\theta}\,\leq\,\tilde
    a_0(\varepsilon)\,\leq\,\tilde A_0 \label{tildea0boundsRefined}, 
 \ee
where $\hat A_0$ and $\tilde A_0$ are positive constants and $0<\theta\leq 2$. 

The restriction \eqref{tildea0boundsRefined} plays an important role, which may be interpreted as that of excluding additional modes which may in principle be localized near the boundary inclusions for small 
$\ve$ and which cannot be accounted for by the limit two-scale problem. The latter possibility was essentially 
excluded by the analysis of Cherdantsev in \cite{cherd}, which has not only established a version 
of the strong two-scale resolvent convergence but also established a key property of two-scale spectral 
compactness (Theorem 6.1 of \cite{cherd}). Namely, if $u^\ve$ is a sequence of normalized eigenvalues 
of the original operator $A_\ve$, see \eqref{eq:eq1new}, with associated eigenvalues 
$\lambda(\ve)\to\lambda_0$, then up to a subsequence $u^\ve$ strongly two-scale converges to a 
(non-zero) $u^{(0)}(x,y)$ which is an eigenfunction of the limit two-scale operator $A_0$ for which 
$\lambda_0$ is the eigenvalue. (We emphasize that 
for establishing this key result Cherdantsev followed again the general strategy of Zhikov, which 
however required for him to develop a key novel technical ingredient of a uniform exponential decay 
of the eigenfunctions of $A_\ve$, see Theorem 3.1 of \cite{cherd}.) 

The above result of \cite{cherd}, when applied to our setting, immediately implies the following. 
Given $\lambda_0$ an isolated 
eigenvalue of $A_0$ of a finite multiplicity $m$ with $\beta(\lambda_0)<0$ and 
$\lambda_0\neq\lambda_j$, $j\geq 0$, for small enough $\ve$ there will be exactly $m$ eigenvalues of 
$A_\ve$ (counted with their multiplicities) near $\lambda_0$. (Notice that this is formally valid for 
$m=0$ as well: if $\lambda_0$ is {\it not} an eigenvalue then, for small $\ve$, there are no eigenvalues of 
$A_\ve$ near $\lambda_0$.) 
 This in turn implies that in \eqref{eq28}, 
for small enough $\ve$, $J_\ve=m$, which can be interpreted as an $\ve^{1/2}$-error bound not only 
on the eigenvalues but also on the eigenfunctions. 
We formulate this refined result as the following theorem. 

\begin{theorem}\label{refinedtheorem}
Let all the assumptions of Theorem \ref{thm1} be satisfied, and let additionally the 
assumption \eqref{tildea0bounds} be strengthened to \eqref{tildea0boundsRefined}, and let 
$\lambda_0$ be an isolated eigenvalue of the limit operator $A_0$ of a finite multiplicity $m\geq 0$. 
Then there exists $\ve_0>0$,  and 
constants $\delta>0$ and $C_1>0$ independent of $\varepsilon$  such that
for any $0<\ve\leq \ve_0$ 
there exist, within the $\delta$-neighbourhood of $\lambda_0$ ($\{\lambda:\, |\lambda-\lambda_0|<\delta\}$), exactly $m$ 
eigenvalues $\lambda^{(j)}(\varepsilon)$ of operator $A_\varepsilon$ counted with their multiplicities, 
such that
\be
|\lambda^{(j)}(\varepsilon)-\lambda_0|\,\leq\,C_1\varepsilon^{1/2}, \ \ \ j=1,...,m.
\label{mainresultref}
\end{equation}

Moreover if $u^0(x,y)=(u_0(x),v(x,y))$ is an eigenfunction of $A_0$ which corresponds to $\lambda_0$ then  
there exist constants $c_j$, $j=1,...,m$, such that
\begin{equation}
\left\| u^{0}(x,x/\ve)-\sum_{j=1}^m 
c_ju_j^\varepsilon\right\|_{L_2(\mathbb
R^n)}\,\leq\,C_2\varepsilon^{1/2}, \label{eq28ref}
\end{equation}  
where $u^\varepsilon_j(x)$ are $L_2$-normalized eigenfunctions associated with $\lambda^{(j)}(\ve)$, and the
constant $C_2$ is  independent of $\varepsilon$.
\end{theorem} 

In particular, if $m=1$ i.e. $\lambda_0$ is a simple isolated eigenvalue of $A_0$, then 
\begin{equation} 
\left\Vert
u^\varepsilon(x)-u^0(x,x/\varepsilon)\right\Vert_{L^2({\mathbb R}^n}\,\leq\,C\,\varepsilon^{1/2}, \ \ 
\left|\lambda(\varepsilon)-\lambda_0\right|\,\leq\,C_1\varepsilon^{1/2}, 
\label{conveigfns}
\end{equation}
i.e. there holds an $\ve^{1/2}$-order $L^2$-error bound between the exact and the approximate eigenfunctions,  
as well as the eigenvalues. 


Finally, let us remark that we expect the conditions on the regularity of the
boundaries of both the periodic inclusion $Q_0$ and the defect $\Omega_2$ could also be relaxed.  
Indeed, while assuming the $C^\infty$-smoothness of the above domain,  the presented argument has only used 
up to $C^2$-regularity. Furthermore, in principle, the argument could be refined to relax the 
regularity restrictions further, 
cf. e.g. \cite{Zh05}, \cite{JKO} \S 8.3.

\appendix

\section{Formal derivation of the limit problem
(\ref{330})--(\ref{limsystem}).}
\setcounter{equation}{0}
\renewcommand{\theequation}{\mbox{\thesection.\arabic{equation}}}

This appendix formally derives the limit equations
\eqref{330}--\eqref{limsystem} via two-scale asymptotic expansions
in the form \eqref{ansatz1}. It also establishes the structure of
the 
corrector terms in \eqref{ansatz1} as
required for subsequent rigorous justification.

The asymptotic solution to \eqref{330}--\eqref{limsystem} is sought in the form
of two-scale ansatz \eqref{ansatz1}--\eqref{ansatz2}, where
$u^{(0)}(x,y)$, $u^{(1)}(x,y)$ and $u^{(2)}(x,y)$ are assumed to depend
periodically on the ``fast'' variable $y$ only outside the defect $\Omega_2$.
The exact solution $u^\ve(x)$ satisfies the standard continuity conditions at
the boundary $\partial \Omega_0^\ve$ of the small inclusions away from the defect:
\be
\left(u^\ve(x)\right)_0\,=\,\left(u^\ve(x)\right)_1, \,\,\,x\in \partial \Omega_0^\ve,
\label{ucont}
\ee
and
\be
\left(a_0\ve^2\frac{\partial u^\ve}{\partial n}(x)\right)_0\,=\,
\left(a_1\frac{\partial u^\ve}{\partial n}(x)\right)_1,  \,\,\,x\in \partial \Omega_0^\ve,
\label{ducont}
\ee
where the subscripts ``0'' and ``1'' denote the limit values at the boundary
evaluated in $\Omega_0^\ve$ and $\Omega_1^\ve$, respectively; $n$ stands for
unit normal to $\partial \Omega_0^\ve$ which we select as outward for the matrix phase
$\Omega_1^\ve$ (and hence inward for the inclusions $\Omega_0^\ve$).
Similar boundary conditions are also satisfied at the boundaries of the defect
$\partial\Omega_2$ and of the ``boundary layer'' inclusions $\tilde\Omega_0^\ve$.

The ansatz \eqref{ansatz1}--\eqref{ansatz2} is first substituted into the equation
\eqref{eq:eq1}--\eqref{eq:eq1new} and the interface conditions 
\eqref{ucont}-\eqref{ducont}, away from the defect.

By equating first the terms of order $\ve^{-2}$ in \eqref{eq:eq1} and
of order $\ve^{-1}$ in \eqref{ducont},
\[
-\nabla_y\cdot a_1\nabla_yu^{(0)}(x,y)=0, \  \,\, y \in  Q_1; \,\,\,
a_1\frac{\partial}{\partial n_y}u^{(0)}(x,y)=0, \,\,y \in \partial
Q_0,
\]
with $n$ being the outward unit normal for $Q_1$, and $\partial/\partial n_y:=n\cdot\nabla_y$ is the 
normal derivative ``in $y$''.

This is a homogeneous Neumann problem in $Q_1$ with periodic  boundary conditions,
whose solution is an arbitrary constant in $Q_1$ (i.e. is independent of $y$) which
implies \eqref{u0}. The balance of the terms of order $\ve^0$ in \eqref{ucont}
implies (cf. \eqref{332})
\be
v(x,y)=0,\,\,\, \ \ y\in \partial Q_0.
\label{vbc}
\ee

Equating next the terms of order $\ve^{-1}$ in \eqref{eq:eq1} and
of order $\ve^{0}$ in \eqref{ducont}, we arrive at:
\[
-\nabla_y\cdot a_1\nabla_yu^{(1)}(x,y)=0, \,\, y \in  Q_1; \,\,\, \ \ 
a_1\frac{\partial}{\partial
n_y}u^{(1)}(x,y)\,=\,-\,a_1\frac{\partial}{\partial n_x}u_0(x),
 \,\,y\in \partial Q_0,
\]
$\partial/\partial n_x:=n\cdot\nabla_x$, 
together with the periodicity conditions in $y$. This is a standard
corrector problem for ``soft'' inclusions (or perforated) periodic
domains. As a result, \be u^{(1)}(x,y)\,=\,N_j(y)\frac{\partial
u_0(x)}{\partial x_j}, \label{u1} \ee where $N_j$ is the solution
of the ``soft'' unit cell problem with periodic boundary
conditions (e.g. \cite{JKO} \S 3.1):
 \be
    a_1\triangle N_j=0,\ y\in  Q_1;\ \ \
    a_1\frac{\partial}{\partial n_y} N_j\,=\,-\,a_1n_j, \,y\in  \partial Q_0.
 \label{eq:N}
 \ee
It is convenient for some subsequent analysis to view the functions $N_j(y)$ as defined in the whole of
the periodicity cell $Q$: with this aim they are extended into the inclusion domain
$Q_0$  by, for example, harmonic continuation: $\triangle N_j=0,\ \text{in}\  Q_0$,
$N_j\in C(Q)$. The above procedure determines $N_j$ up to a constant which can be 
specified by the condition that the average of $N_j$ over $Q$ is zero:
$\langle N_j(y)\rangle_{y}=0$.

Equate now the terms of order $\ve^{0}$ in
\eqref{eq:eq1}. As a result, in $Q_0$, \be
-a_0\Delta_yv(x,y)\,=\,\lambda_0(u_0(x) + v(x,y)), \label{u20} \ee
which together with \eqref{vbc} fully recovers (\ref{332}). In
particular, 
assuming $\lambda_0\neq\lambda_j$, $j\geq 1$, i.e. $\lambda_0$ is not a Dirichlet 
eigenvalue of $-\Delta_y$ in $Q_0$, 
this implies that $v(x,y)$ can be uniquely presented in the form \eqref{Vy}, 
$v(x,y)=u_0(x)V(y)$, where 
where $V$ is a solution of \eqref{eq:V}, where it is further assumed that $V$ is extended 
by zero to $Q$ and is then periodically extended on the whole $\mathbb{R}^n$.

In turn, in $Q_1$, taking into account \eqref{u0} and \eqref{u1}, the balance of order $\ve^0$ terms 
in \eqref{eq:eq1} yields:
 \be
     -a_1\Delta_yu^{(2)}(x,y)\,=\, 2a_1N_{j,k}(y)u_{0,jk}(x)\,
    +\,a_1\Delta_xu_0(x)\,+\,\lambda_0u_0(x), \,\, y\in  Q_1.
 \label{u21}
 \ee
(Henceforth comma in subscript denotes differentiation with respect to variables with 
following indices.) 

 This equation has to be supplemented by boundary conditions which
result from equating in \eqref{ducont} the terms of order $\ve^1$,
yielding 
 \be
    a_1\frac{\partial}{\partial
    n_y}u^{(2)}\,=\,-\,a_1\,N_j\frac{\partial}{\partial n_x}u_{0,j}
    +a_0\frac{\partial}{\partial n_y}v, \ \ \ y\in\partial Q_0.
 \label{u2bc} \ee

Treating \eqref{u21}--\eqref{u2bc} as boundary value problem for
$u^{(2)}$ in $y$ for any fixed $x$, the Green's formula together with the
periodicity boundary conditions in $y$ imply:
\[
a_1\int_{Q_1}\Delta_yu^{(2)}= a_1\int_{\partial
Q_0}\frac{\partial}{\partial n_y}u^{(2)},
\]
which yields:
 \[
    \left(-a_1\,\Delta_xu_0(x)-\lambda_0u_0(x)\right)|Q_1|-2a_1u_{0,jk}(x)\int_{Q_1}N_{j,k}(y)dy=
 \]
 \[ \int_{\partial Q_0} \left(-a_1\,N_j(y)u_{0,jk}(x)n_k(y)
    +a_0\frac{\partial}{\partial n_y}v(x,y)\right)dy.
\]
Applying the integration by parts to the right hand side surface integrals and
using \eqref{u20} we arrive at \eqref{331} where
 \be
    A^{\rm hom}_{ij}:=\biggl\langle\mu(y) a_1\biggl(\delta_{ij}+
        N_{j,i}(y)\biggr)\biggr\rangle_{y},
 \label{eq:A}
 \ee
and $\mu(y)$ is the characteristic function of $Q_1$ ($\mu(y)=1$, $y\in Q_1$;
$\mu(y)=0$, $y\in Q_0$). This is a well-known representation of the entries
of the homogenized matrix $A^{\rm hom}$ in a perforated domain (with ``holes'' in soft inclusions $Q_0$), equivalent to
\eqref{eq:ahom}, see e.g. \cite{JKO} \S 3.1. Hence the limit equation
\eqref{331} is recovered.

To complete the formal derivation of the limit problem \eqref{330}--(\ref{limsystem})
the natural equation \eqref{330} is simply postulated within the
(homogeneous) defect $\Omega_2$. The limit interface conditions \eqref{limsystem}
at the boundary of the defect are also postulated: they have the meaning of the
continuity of the fields and the flows ``to main order'' and the proof of the Theorem
\ref{l:estim} ensures a posteriori that those produce a controllably small boundary layer,
ultimately ensuring the main  result \eqref{mainresult} of the paper.

\subsection*{Acknowledgements}
A preliminary version of this work was completed back in 2006, and was then 
supported by Bath Institute for Complex Systems 
(BICS, EPSRC grant GR/S86525/01). 
The present manuscript represent a substantially refined and updated version of the one 
initially published in the BICS preprints series (2006). 
The authors have been grateful to Vladimir 
Kamotski, R.R. Gadyl'shin, V.V. Zhikov for discussions and useful comments. Thanks 
are also due to D.M. Bird for stimulating discussions on some physical 
motivations.

\end{document}